\documentclass[10pt]{article}

\usepackage{amssymb}
\usepackage{amsfonts}
\usepackage{amsmath,amsthm}
\usepackage{graphics}
\usepackage[table]{xcolor}
\usepackage{booktabs}
\usepackage{graphicx}
\usepackage{array}
\usepackage{subfigure}
\usepackage{rotating}

\oddsidemargin 0cm \numberwithin{equation}{section}
\newtheorem{theorem}{Theorem}[section]
\newtheorem{definition}[theorem]{Definition}
\newtheorem{rem}[theorem]{Remark}
\newtheorem{cor}[theorem]{Corollary}
\newtheorem{lem}[theorem]{Lemma}
\newtheorem{pro}[theorem]{Proposition}

\renewcommand{\v}[1]{\ensuremath{\mathbf{#1}}}

\usepackage{color, graphicx}
\usepackage{mathrsfs,amsfonts,latexsym,amssymb,amsthm,amsmath}
\usepackage{anysize}\marginsize{1.5cm}{1.5cm}{1cm}{2.5cm}

\usepackage{graphicx}

\usepackage{array}
\newcolumntype{C}[1]{>{\centering\arraybackslash}p{#1}}

\usepackage[font=footnotesize]{caption}
\usepackage{mathabx}
\usepackage{url}
\usepackage{xspace}

\usepackage{authblk}

\usepackage[normalem]{ulem}

\usepackage{hhline}
\usepackage{makecell}

\usepackage{graphicx}
\usepackage{multirow}

\usepackage[font=footnotesize]{caption}
\usepackage{bm}
\usepackage{tcolorbox}
\usepackage{framed}
\usepackage{enumitem}
\usepackage{relsize}        
\usepackage{url}            
\usepackage{color}\usepackage{graphicx}\usepackage{mathrsfs,amsfonts,latexsym,amssymb,amsthm,amsmath}

\numberwithin{equation}{section}

\usepackage{hhline}
\usepackage{makecell}
\usepackage{graphicx}
\usepackage{multirow}
\usepackage{bm}
\usepackage{tcolorbox}
\usepackage{framed}
\usepackage{url}            
\usepackage{color}
\usepackage{graphicx}
\usepackage{mathrsfs,amsfonts,latexsym,amssymb,amsthm,amsmath}

\date{}
\begin{document}

\title{Spectral analysis and multigrid preconditioners for two-dimensional
space-fractional diffusion equations}

\author[1]{Hamid Moghaderi}
\author[1]{Mehdi Dehghan}
\author[2]{Marco Donatelli}
\author[3]{Mariarosa Mazza}
\affil[1]{\small Department of Applied
Mathematics, Faculty of Mathematics and Computer Science, Amirkabir
University of Technology, No. 424, Hafez Ave., 15914, Tehran, Iran}
\affil[2]{\small Department of Science and High Technology, University of Insubria, Via Valleggio 11, 22100 Como, Italy}
\affil[3]{\small Division of Numerical Methods for Plasma Physics, Max Planck Institute for Plasma Physics, Boltzmannstrasse 2, 85748 Garching bei M\"unchen, Germany}

\maketitle

\begin{abstract}

Fractional diffusion equations (FDEs) are a mathematical tool used for describing some special diffusion phenomena arising in many different applications like porous media and computational finance. In this paper, we focus on a two-dimensional space-FDE problem discretized by means of a second order finite difference scheme obtained as combination of the Crank-Nicolson scheme and the so-called weighted and shifted Gr\"unwald formula.

By fully exploiting the Toeplitz-like structure of the resulting linear system, we provide a detailed spectral analysis of the coefficient matrix at each time step, both in the case of constant and variable diffusion coefficients. Such a spectral analysis has a very crucial role, since it can be used for designing fast and robust iterative solvers. In particular, we employ the obtained spectral information to define a Galerkin multigrid method based on the classical linear interpolation as grid transfer operator and damped-Jacobi as smoother, and to prove the linear convergence rate of the corresponding two-grid method. The theoretical analysis suggests that the proposed grid transfer operator is strong enough for working also with the V-cycle method and the geometric multigrid. On this basis, we introduce two computationally favourable variants of the proposed multigrid method and we use them as preconditioners for Krylov methods. Several numerical results confirm that the resulting preconditioning strategies still keep a linear convergence rate.

\end{abstract}

\noindent\textbf{Keywords: }fractional diffusion equations; CN-WSGD scheme; spectral analysis; GLT theory; multigrid methods

\let\thefootnote\relax\footnote{\emph{Email addresses}: \texttt{hmoghaderi@yahoo.com; hamidm91@aut.ac.ir} (Hamid Moghaderi), \texttt{mdehghan@aut.ac.ir; mdehghan.aut@gmail.com} (Mehdi Dehghan), \texttt{\texttt{marco.donatelli@insubria.it}} (Marco Donatelli), \texttt{\texttt{mariarosa.mazza@ipp.mpg.de}} (Mariarosa Mazza)}


\section{Introduction}\label{Introduction}
Fractional diffusion equations (FDEs) are a special class of partial differential equations (PDEs) used for describing subdiffusion and superdiffusion phenomena. Among their applications: finance, biology, turbulent flow, and image processing \cite{fin,bio,turb,imag}. In more detail, a standard diffusion equation becomes a \emph{time-FDE} when the first order derivative in time is replaced by a fractional derivative in the Caputo form (see e.g., \cite{frac} for a precise definition) and/or a \emph{space-FDE} when the second order derivatives in space are replaced by the fractional Riemann-Liouville derivatives (to be defined below). In this paper, we deal with space-FDEs and we refer to them simply as FDEs.

Closed-form solutions for FDEs are rarely available, then various numerical discretization methods have been proposed in the last decades. Among them, finite difference schemes have been widely studied. For instance, in 2004 and 2006, Meerschaert and Tadjeran proposed a first order accurate finite difference scheme obtained combining an implicit Euler method in time with a first order approximation of the space derivatives called shifted Gr\"unwald formula \cite{24,25}. In 2006, together with Scheffler the same authors proposed a finite difference discretization method based on the classical Crank-Nicolson (CN) and the Richardson extrapolation, which guarantee the second order accuracy in time and space, respectively \cite{scheffler}. Both methods were proved to be consistent and unconditionally stable. More recently, in \cite{39} the authors, inspired by the shifted Gr\"unwald difference operator and the CN technique, defined a more general and flexible class of second order accurate methods combining the CN discretization in time with a second order approximation of the Riemann-Liouville fractional derivatives called weighted and shifted Gr\"unwald difference (WSGD). Such methods are briefly referred to as CN-WSGD.

From a numerical linear algebra viewpoint, it is worth noticing that, if one of the previously described finite difference schemes is chosen, then the resulting linear system shows a strong structure and indeed the related coefficient matrices can be seen as a sum of diagonal times Toeplitz matrices (see e.g., \cite{WWT}). As a consequence, the storage requirement is reduced from $\mathcal{O}(N^2)$ to $\mathcal{O}(N)$ and the complexity of the matrix-vector product from $\mathcal{O}(N^2)$ to $\mathcal{O}(N\log N)$, where $N$ is the mesh-size at each time step.

Recently, several fast iterative solvers which exploit the aforementioned Toeplitz-like structure have been proposed for solving linear systems coming from a finite difference discretization of one-dimensional FDEs. Among them, a circulant preconditioning strategy for Conjugate Gradient for Normal Residual method (CGNR) \cite{15}, and an efficient multigrid method \cite{29}. The former has been proven to be superlinearly convergent when the diffusion
coefficients are constant, the latter has been shown to be optimal when they are also equal.
In paper \cite{M2}, the authors recognize in the Toeplitz-like structure of the FDEs linear systems a Generalized Locally Toeplitz (GLT) sequence and then use the GLT machinery to study its singular value/eigenvalue distribution as the matrix size diverges. The obtained spectral information is employed to propose two competitive tridiagonal structure preserving preconditioners for Krylov methods and to show that the superlinear convergence of the CGNR method with the circulant preconditioner in \cite{15} cannot be replicated by any Krylov method in the nonconstant coefficient case, while the multigrid in \cite{29} is optimal also for nonconstant coefficients under the only hypothesis that they are uniformly bounded and positive. 

Of course, fast solvers for multidimensional FDE problems are of crucial interest. In the two-dimensional setting, we mention the ones proposed in \cite{WW,M1,LOD}. The first reference deals with an alternating-direction finite difference method for two-dimensional FDEs by fully decoupling the linear systems in $x$-direction and $y$-direction. In \cite{M1} a preconditioner defined by the incomplete LU factorization of a block-banded-banded-block approximation of the coefficient matrix is introduced, and the linear systems are solved by a preconditioned Generalized Minimal Residual (GMRES) method and a preconditioned CGNR method. In \cite{LOD} the authors propose a Locally One Dimensional (LOD) finite difference method for two-dimensional (and three-dimensional) Riesz FDE problem and use the LOD-multigrid method to solve the resulting linear systems. As for the one-dimensional case, also in the multidimensional setting the spectral analysis of the coefficient matrix is crucial to fully understand the behaviour of preconditioned
Krylov and multigrid methods when used for solving the associated linear system. However, for multidimensional FDEs such a spectral study is missing even in the case of constant diffusion coefficients.

In this paper, we focus on a two-dimensional FDE problem discretized by means of the CN-WSGD scheme, and we provide a detailed spectral analysis of the resulting coefficient matrix at each time step, both in the constant and variable diffusion coefficients case. The obtained spectral information suggests that classical preconditioners based on multilevel circulant matrices are not suited for solving the involved linear systems, while multigrid methods can be effective and robust solvers. On this basis, we propose a Galerkin multigrid method with classical linear interpolation as grid transfer operator and damped-Jacobi as smoother, and we prove that the corresponding two-grid method has a linear convergence rate. More precisely, we compute the spectral symbol of the involved matrix-sequences both in the constant and variable coefficients case and use this knowledge to formally prove the convergence and the optimality of the two-grid method, i.e., we prove that the number of iterations required to reach a prescribed accuracy depends neither on the discretization step nor on the order of the fractional derivatives. The analysis of the symbol is also used for defining the parameter of the damped Jacobi such that it satisfies the smoothing property \cite{20}.

Concerning the V-cycle, its linear convergence rate has been proven only for sequences of matrices in some trigonometric algebras \cite{222,0}. In spite of this, we show that the linear interpolation operator reveals powerful enough to work also under some perturbations and then we introduce two computationally favourable variants of the proposed multigrid, especially suited in the variable coefficients case. Indeed, in such a case, the setup phase of a multigrid based on the Galerkin approach is computationally too expensive because the structure of the coefficient matrix cannot be preserved, and this prevents a fast computation of the matrix-vector product. Therefore, we propose two preconditioning alternatives:
\begin{enumerate}
\item[1)] a Galerkin multigrid applied to a band preconditioner built using the Laplacian;
\item[2)] a geometric multigrid applied to the full coefficient matrix.
\end{enumerate}
The first strategy is computationally very attractive because it involves matrices with only five nonzero diagonals and hence has a linear cost, with a small constant, in $N$. However, its effectiveness reduces when both the fractional derivative orders are far from $2$. On the contrary, the second strategy is reliable for every fractional order, does not need the setup phase and is such that each iteration has a computational cost only about 4/3 times higher than the Jacobi smoother (see \cite{Trot}). Moreover, it is also well suited as a stand alone solver.

The paper is briefly summarized as follows. In Section~\ref{CN-WSGD} we introduce the two-dimensional FDEs equations and recall the CN-WSGD scheme. In Section~\ref{sec:spectr} we firstly describe the notion of symbol and of spectral distribution of a matrix-sequence, as well as the idea behind the GLT. Later, we use these tools to retrieve spectral information on the involved matrices. Such spectral information is then used in Section~\ref{sec:mgm} to study the convergence and the optimality of the proposed multigrid. Finally, Section~\ref{Example} is devoted to numerical examples and Section~\ref{Conclusions} contains conclusions and open problems.
\section{Problem setting: the CN-WSGD scheme for 2D space-FDEs}\label{CN-WSGD}
In this paper, we are interested in the following initial-boundary
value problem of two-dimensional space-FDE
\begin{align}\label{1}
\begin{cases}
\frac{\partial u(x,y,t)}{\partial
t}=d_+(x,y,t)\frac{\partial^{\alpha}
u(x,y,t)}{\partial_{+}x^{\alpha}}+d_-(x,y,t)\frac{\partial^{\alpha}
u(x,y,t)}{\partial_{-}x^{\alpha}}&\\
\qquad \qquad + \, e_+(x,y,t)\frac{\partial^{\beta}
u(x,y,t)}{\partial_{+}y^{\beta}}+e_-(x,y,t)\frac{\partial^{\beta}
u(x,y,t)}{\partial_{-}y^{\beta}}+v(x,y,t), \qquad \qquad &(x,y,t)\in\Omega\times(0,T],\\
u(x,y,t)=0,  &(x,y,t)\in\partial\Omega\times[0,T],\\
u(x,y,0)=u_0(x,y), &(x,y)\in\bar{\Omega},
\end{cases}
\end{align}
where $\Omega=(a_1,b_1)\times(a_2,b_2)$ is the space domain, and $\alpha,\beta\in(1,2)$ are
the fractional derivative orders with respect to $x$ and $y$,
respectively. The nonnegative functions $d_\pm(x,y,t)$ and
$e_\pm(x,y,t)$ are the diffusion coefficients and $v(x,y,t)$ is the
forcing term. The left-sided $(+)$ and the right-sided $(-)$
fractional derivatives in (\ref{1}) are defined in Riemann-Liouville
form as follows
\begin{align*}
\frac{\partial^{\alpha}
u(x,y,t)}{\partial_{+}x^{\alpha}}=\frac{1}{\Gamma(2-\alpha)}
\frac{d^2}{dx^2}\int_{a_1}^x
\frac{u(\xi,y,t)}{(x-\xi)^{\alpha-1}}d\xi,
& \qquad \qquad
\frac{\partial^{\alpha}
u(x,y,t)}{\partial_{-}x^{\alpha}}=\frac{1}{\Gamma(2-\alpha)}
\frac{d^2}{dx^2}\int_x^{b_1}
\frac{u(\xi,y,t)}{(\xi-x)^{\alpha-1}}d\xi,\\
\frac{\partial^{\beta}
u(x,y,t)}{\partial_{+}y^{\beta}}=\frac{1}{\Gamma(2-\beta)}
\frac{d^2}{dy^2}\int_{a_2}^y
\frac{u(x,\eta,t)}{(y-\eta)^{\beta-1}}d\eta,
& \qquad \qquad
\frac{\partial^{\beta}
u(x,y,t)}{\partial_{-}y^{\beta}}=\frac{1}{\Gamma(2-\beta)}
\frac{d^2}{dy^2}\int_y^{b_2}
\frac{u(x,\eta,t)}{(\eta-y)^{\beta-1}}d\eta.
\end{align*}
For the discretization of the FDE problem (\ref{1}) we apply the second order accurate CN-WSGD scheme (see \cite{39}). In order to introduce such a scheme, let us fix three positive integers $M,~n_1,~n_2$ and discretize the domain $\Omega\times[0,T]$ with
\begin{eqnarray*}
x_i=a_1+ih_x,&\quad h_x=\frac{(b_1-a_1)}{n_1+1},&\quad i=0,1,\dots,n_1,\\
y_j=a_2+jh_y,&\quad h_y=\frac{(b_2-a_2)}{n_2+1},&\quad j=0,1,\dots,n_2,\\
t^{(m)}=m\triangle t,& \;\triangle t=\frac{T}{M},\quad &\quad m=0,1,\dots,M,\\
t^{(m-1/2)}:=\frac{t^{(m)}+t^{(m-1)}}{2},&&\quad m=1,2,\dots,M.
\end{eqnarray*}
Let us define
\begin{align*}
{}_L\mathcal{D}^{\alpha}_{h_x}u(x_i,y_j,t^{(m)})=
\frac{1}{h_x^{\alpha}}\sum\limits_{k=0}^{i}w_k^{(\alpha)}u(x_{i-k+1},y_j,t^{(m)}),
& \qquad
{}_R\mathcal{D}^{\alpha}_{h_x}u(x_i,y_j,t^{(m)})=
\frac{1}{h_x^{\alpha}}\sum\limits_{k=0}^{n_1-i+1}w_k^{(\alpha)}u(x_{i+k-1},y_j,t^{(m)}),\\
{}_L\mathcal{D}^{\beta}_{h_y}u(x_i,y_j,t^{(m)})=
\frac{1}{h_y^{\beta}}\sum\limits_{k=0}^{j}w_k^{(\beta)}u(x_i,y_{j-k+1},t^{(m)}),
& \qquad
{}_R\mathcal{D}^{\beta}_{h_y}u(x_i,y_j,t^{(m)})=
\frac{1}{h_y^{\beta}}\sum\limits_{k=0}^{n_2-j+1}w_k^{(\beta)}u(x_i,y_{j+k-1},t^{(m)}).
\end{align*}
Using the WSGD formula with shifting parameters $(p,q)=(1,0)$, we obtain the following expression for the fractional derivatives
\begin{align*}
\frac{\partial^{\alpha} u(x_i,y_j,t^{(m)})}{\partial_{+}x^{\alpha}}=
{}_L\mathcal{D}^{\alpha}_{h_x}u(x_i,y_j,t^{(m)}) +\mathcal{O}(h_x^2),
& \qquad
\frac{\partial^{\alpha} u(x_i,y_j,t^{(m)})}{\partial_{-}x^{\alpha}}=
{}_R\mathcal{D}^{\alpha}_{h_x}u(x_i,y_j,t^{(m)}) +\mathcal{O}(h_x^2),\\
\frac{\partial^{\beta} u(x_i,y_j,t^{(m)})}{\partial_{+}y^{\beta}}=
{}_L\mathcal{D}^{\beta}_{h_y}u(x_i,y_j,t^{(m)})+\mathcal{O}(h_y^2),
& \qquad
\frac{\partial^{\beta} u(x_i,y_j,t^{(m)})}{\partial_{-}y^{\beta}}
={}_R\mathcal{D}^{\beta}_{h_y}u(x_i,y_j,t^{(m)})+\mathcal{O}(h_y^2).
\end{align*}
where $u(a_1,y,t)=u(x,a_2,t)=u(b_1,y,t)=u(x,b_2,t)=0$, because of the Dirichlet boundary conditions. Here the coefficients $w_k^{(\gamma)}$ are given by
\begin{eqnarray}\label{w}
w_0^{(\gamma)}=\frac{\gamma}{2}g_0^{(\gamma)},~~~~~w_k^{(\gamma)}=\frac{\gamma}{2}g_k^{(\gamma)}+\frac{2-\gamma}{2}g_{k-1}^{(\gamma)},~~~k\geq1,
\end{eqnarray}
where $g_k^{(\gamma)}$ are the alternating fractional binomial
coefficients defined as
\begin{equation}\label{eq:gk}
g_k^{(\gamma)}=(-1)^k
\binom{\gamma}{k}=\frac{(-1)^k}{k!}\gamma(\gamma-1)\dots(\gamma-k+1), \qquad k=0,1,\dots,
\end{equation}
with the formal notation $\binom{\gamma}{0}=1$. For
$\gamma=\alpha,\beta$, the fractional binomial coefficients
$g_k^{(\gamma)}$ have the following properties
\begin{align*}
\begin{cases}
g_0^{(\gamma)}=1, \quad g_1^{(\gamma)}=-\gamma, &g_2^{(\gamma)}>g_3^{(\gamma)}>\dots>0,\\
\sum\limits_{j=0}^{\infty}g_j^{(\gamma)}=0, &\sum\limits_{j=0}^{n}g_j^{(\gamma)}<0, \; n\geq 1,\\
g_j^{(\gamma)}=\mathcal{O}(j^{-(\gamma+1)}),&\\
\end{cases}
\end{align*}
proved in \cite{24,25,29}.

Similarly, the coefficients $w_k^{(\gamma)}$ satisfy few properties,
summarized in the following proposition (see \cite{39}).
\begin{pro}\label{ww}
Let $w_k^{(\gamma)}$, $1<\gamma<2$ be defined as in (\ref{w}).
Then the coefficients $w_k^{(\gamma)}$ satisfy the following properties
\[
\begin{cases}
w_0^{(\gamma)}=\frac{\gamma}{2}, \quad w_1^{(\gamma)}=\frac{2-\gamma-\gamma^2}{2}<0, \quad
w_2^{(\gamma)}=\frac{\gamma(\gamma^2+\gamma-4)}{4}, \quad
1\geq w_0^{(\gamma)}\geq w_3^{(\gamma)}\geq w_4^{(\gamma)} \dots \geq  0,\\
\sum\limits_{j=0}^{\infty}w_j^{(\gamma)}=0, \quad
\sum\limits_{j=0}^{n}w_j^{(\gamma)}<0,\;n\geq 2,\\
w_j^{(\gamma)}=\mathcal{O}(j^{-(\gamma+1)}).
\end{cases}
\]
\end{pro}
The CN-WSGD scheme is obtained combining the CN scheme in time with the WSGD formula for the fractional derivatives and can be written as follows
\begin{eqnarray}\label{CNWSGD}
(1-\delta_x^{\alpha,(m)}-\delta_y^{\beta,(m)})u_{i,j}^{(m)}=(1+\delta_x^{\alpha,(m-1)}+\delta_y^{\beta,(m-1)})u_{i,j}^{(m-1)}+\triangle
t v_{i,j}^{(m-1/2)},
\end{eqnarray}
where $u_{i,j}^{(m)}\approx u(x_i,y_j,t^{(m)})$, $v_{i,j}^{(m-1/2)}=v(x_i,y_j,t^{(m-1/2)})$, and
\[
\delta_x^{\alpha,(m)}=\frac{d_{i,j}^{+,(m)}\triangle
t}{2}{}_L\mathcal{D}^{\alpha}_{h_x}+\frac{d_{i,j}^{-,(m)}\triangle
t}{2}{}_R\mathcal{D}^{\alpha}_{h_x}, \qquad
\delta_y^{\beta,(m)}=\frac{e_{i,j}^{+,(m)}\triangle
t}{2}{}_L\mathcal{D}^{\beta}_{h_y}+\frac{e_{i,j}^{-,(m)}\triangle
t}{2}{}_R\mathcal{D}^{\beta}_{h_y},
\]
with
\[
d_{i,j}^{\pm,(m)}=d_\pm(x_i,y_j,t^{(m)}), \qquad e_{i,j}^{\pm,(m)}=e_\pm(x_i,y_j,t^{(m)}).
\]

Now, we can write the matrix form of the above discretization. First,
we need to introduce the following notations.
Let $N=n_1n_2$ and define the following objects:
\begin{enumerate}
    \item The $N$-dimensional vectors
\begin{align*}
\mathbf{u}^{(m)}&=[u_{1,1}^{(m)},\dots,u_{n_1,1}^{(m)}, u_{1,2}^{(m)},\dots,u_{n_1,2}^{(m)},\dots,u_{1,n_2}^{(m)},\dots,u_{n_1,n_2}^{(m)}]^T,\\
\mathbf{d}^{(m)}_\pm&=[d_{1,1}^{\pm,(m)},\dots,d_{n_1,1}^{\pm,(m)},d_{1,2}^{\pm,(m)},\dots,d_{n_1,2}^{\pm,(m)},\dots,d_{1,n_2}^{\pm,(m)},\dots,d_{n_1,n_2}^{\pm,(m)}]^T,\\
\mathbf{e}^{(m)}_\pm&=[e_{1,1}^{\pm,(m)},\dots,e_{n_1,1}^{\pm,(m)},e_{1,2}^{\pm,(m)},\dots,e_{n_1,2}^{\pm,(m)},\dots,e_{1,n_2}^{\pm,(m)},\dots,e_{n_1,n_2}^{\pm,(m)}]^T,\\
\mathbf{v}^{(m-1/2)}&=[v_{1,1}^{(m-1/2)},\dots,v_{n_1,1}^{(m-1/2)},v_{1,2}^{(m-1/2)},\dots,v_{n_1,2}^{(m-1/2)},\dots,v_{1,n_2}^{(m-1/2)},\dots,v_{n_1,n_2}^{(m-1/2)}]^T.
\end{align*}
    \item The $N\times N$ diagonal matrices
\[
D^{(m)}_\pm=\rm diag(\mathbf{d}^{(m)}_\pm),
\qquad
E^{(m)}_\pm=\rm diag(\mathbf{e}^{(m)}_\pm).
\]
    \item The $\tilde{N} \times \tilde{N}$ Toeplitz matrix 
\begin{equation}\label{eq:AN}
A^{\gamma}_{\tilde{N}}=-
\left(%
\begin{array}{cccccc}
  w_1^{(\gamma)} & w_0^{(\gamma)} & 0 & \cdots & 0 & 0 \\
  w_2^{(\gamma)} & w_1^{(\gamma)} & w_0^{(\gamma)} & \ddots & \ddots & 0 \\
  \vdots & w_2^{(\gamma)} & w_1^{(\gamma)} & \ddots & \ddots & \vdots \\
  \vdots & \ddots & \ddots & \ddots & \ddots & \vdots \\
  w_{\tilde{N}-1}^{(\gamma)} & \ddots & \ddots & \ddots & w_1^{(\gamma)} & w_0^{(\gamma)} \\
  w_{\tilde{N}}^{(\gamma)} & w_{\tilde{N}-1}^{(\gamma)} & \cdots & \cdots & w_2^{(\gamma)} & w_1^{(\gamma)} \\
\end{array}%
\right).
\end{equation}
    \item The $N\times N$ matrices
\begin{equation}\label{AxAy}
A^{(m)}_x=D^{(m)}_+\left(I_{n_2}\otimes
A^{\alpha}_{n_1}\right)+D^{(m)}_-\left(I_{n_2}\otimes
(A^{\alpha}_{n_1})^T\right),\qquad
A^{(m)}_y=E^{(m)}_+\left(A^{\beta}_{n_2}\otimes
I_{n_1}\right)+E^{(m)}_-\left((A^{\beta}_{n_2})^T\otimes I_{n_1}\right),
\end{equation}
where $\otimes$ denotes the usual Kronecker product.
\end{enumerate}
Thus the CN-WSGD method (\ref{CNWSGD}) can be
written in the following matrix form
\begin{eqnarray}\label{system}
\left(I_{N}+rA^{(m)}_x+sA^{(m)}_y\right)\mathbf{u}^{(m)}=\left(I_{N}-rA^{(m-1)}_x-sA^{(m-1)}_y\right)\mathbf{u}^{(m-1)}+\triangle
t \mathbf{v}^{(m-1/2)},
\end{eqnarray}
where $r=\frac{\triangle t}{2h_x^{\alpha}}$, $s=\frac{\triangle t}{2h_y^{\beta}}$, and $I_N$ denotes the $N\times N$ identity matrix.
Multiplying  both sides by $\frac{1}{r}$, we obtain
\[
\bigg(\frac{1}{r}I_{N}+A^{(m)}_x+\frac{s}{r}A^{(m)}_y\bigg)\mathbf{u}^{(m)}=
\bigg(\frac{1}{r}I_{N}-A^{(m-1)}_x-\frac{s}{r}A^{(m-1)}_y\bigg)\mathbf{u}^{(m-1)}+2h_x^{\alpha} \mathbf{v}^{(m-1/2)}.
\]
By defining
\begin{align}
\mathbf{\mathcal{M}}^{(m)}_{(\alpha,\beta),N} &=\frac{1}{r}I_{N}+A^{(m)}_x+\frac{s}{r}A^{(m)}_y,
\label{eq:M}\\
\mathbf{b}^{(m)} &=\bigg(\frac{1}{r}I_{N}-A^{(m-1)}_x-\frac{s}{r}A^{(m-1)}_y\bigg)\mathbf{u}^{(m-1)}+2h_x^{\alpha} \mathbf{v}^{(m-1/2)},\nonumber
\end{align}
the linear system \eqref{system}, which has to be solved at each time step $t^{(m)}$, can be written as
\begin{eqnarray}\label{2}
\mathbf{\mathcal{M}}^{(m)}_{(\alpha,\beta),N}\mathbf{u}^{(m)}=\mathbf{b}^{(m)}.
\end{eqnarray}

Let $\gamma_0=(-1+\sqrt{17})/2\approx 1.562$, then the following results are proved in \cite{39}.
If $\gamma\in[\gamma_0,2)$, from Proposition~\ref{ww} it follows that
$w_2^{(\gamma)}\geq 0$, then the coefficient matrix
$\mathbf{\mathcal{M}}^{(m)}_{(\alpha,\beta),N}$ is a strictly diagonally dominant
M-matrix. In addition,
when the diffusion coefficients are time independent, the
CN-WSGD scheme is unconditionally stable and the truncation error is
$\mathcal{O}(h^2_x +h^2_y +\triangle t^2)$. On the
other side, when $d_+(x,y,t)=d_-(x,y,t)=d$ and
$e_+(x,y,t)=e_-(x,y,t)=e$, with $d$ and $e$ nonnegative constants, then the matrix
$\mathbf{\mathcal{M}}^{(m)}_{(\alpha,\beta),N}$ is a symmetric positive definite
Block-Toeplitz-Toeplitz-Blocks (BTTB) matrix, cf. Definition~\ref{def:toep}. 
\section{Spectral analysis of the coefficient matrix}\label{sec:spectr}
Given the notion of symbol and of spectral distribution in the
eigenvalue and singular value sense, in this section we provide a
spectral analysis of the coefficient matrix-sequence $\{ \mathbf{\mathcal{M}}^{(m)}_{(\alpha,\beta),N} \}_{_{N\in
\mathbb{N}}}$.
In the constant coefficient case, as already
observed in other papers (see e.g., \cite{15}), the coefficient
matrix-sequence is a BTTB sequence:
then using well-known spectral tools for BTTB sequences we determine
its symbol and study its spectral distribution. In the nonconstant
coefficients case, under appropriate conditions, we show that,
$\{\mathbf{\mathcal{M}}^{(m)}_{(\alpha,\beta),N}\}_{_{N\in
\mathbb{N}}}$ belongs to the GLT class and use the GLT machinery to
analyse its singular value/eigenvalue distribution. The resulting
spectral information will be then used in Section \ref{sec:mgm} for the analysis and the design of
numerical solvers to be applied to the considered linear systems.
\subsection{Constant diffusion coefficients case}\label{CDC}
Let us assume that the diffusion coefficients are constant. Under
this condition,
$\{\mathbf{\mathcal{M}}^{(m)}_{(\alpha,\beta),N}\}_{_{N\in
\mathbb{N}}}$ is a sequence of BTTB matrices, or equivalently of 2-level Toeplitz matrices according to the following definition.
\begin{definition}\label{def:toep}
Let $f\in L^1([-\pi,\pi]^d)$ and let
$\{f_\mathbf{k}\}_{_{\mathbf{k}\in \mathbb{Z}^d}}$ be the sequence
of its Fourier coefficients defined as
\begin{eqnarray*}
f_\mathbf{k}:=\frac{1}{(2\pi)^d}\int\limits_{[-\pi,\pi]^d}f(\boldsymbol\theta){\rm
e}^{-\mathbf{i}<\mathbf{k},\boldsymbol\theta>}d\boldsymbol\theta,
\end{eqnarray*}
where $\left\langle { \bf k,\boldsymbol\theta}\right\rangle=\sum_{t=1}^dk_t\theta_t$. Then the \emph{$d$-level Toeplitz matrix} of partial orders $\v{n}=(n_1,n_2,\ldots,n_d)$ associated with $f$ is
\[
T^{(d)}_N(f):=[f_{\v{i}-\v{j}}]_{\v{i},\v{j}=\v{1}}^{\v{n}}
= \left[\cdots\left[\left[f_{i_1-j_1,i_2-j_2,\ldots,i_d-j_d}\right]_{i_d,j_d=1}^{n_d}\right]_{i_{d-1},j_{d-1}=1}^{n_{d-1}}\cdots\right]_{i_1,j_1=1}^{n_1},
\]
where $N=\prod_{i=1}^d n_i$ is the order of the matrix. The function $f$ is called the \emph{symbol} of the matrix-sequence $\{T^{(d)}_{N}(f)\}_N$.
\end{definition}
To clarify the notation for the case $d=2$ of interest, the BTTB matrix of order $N$ associated with $f$ is
\begin{eqnarray*}
T^{(2)}_N(f)=\bigg[\bigg[f_{[i_1-j_1,i_2-j_2]}\bigg]_{_{i_1,j_1=1}}^{n_1}\bigg]_{_{i_2,j_2=1}}^{n_2},
\end{eqnarray*}
or equivalently,
\begin{eqnarray*}
T^{(2)}_N(f)=\sum\limits_{|j_1|\leq n_1}\sum\limits_{|j_2|\leq n_2}
f_{[j_1,j_2]}J_{n_1}^{[j_1]}\otimes J_{n_2}^{[j_2]},
\end{eqnarray*}
where $J_{n_i}^{[j_i]}\in \mathbb{R}^{n_i\times n_i}$ are
matrices whose entry $(s,t)$-th equals $1$ if $s-t=j_i$ and is $0$
elsewhere.

When $d=1$, i.e., for Toeplitz matrices, we simplify the notation using
\[T_N(f):=T^{(1)}_N(f).\]

\begin{definition}
The \emph{Wiener class} is the set of functions
\begin{eqnarray*}
f(\boldsymbol{\theta})=\sum\limits_{\v{k}\in \mathbb{Z}^d} f_{\v{k}}{\rm
e}^{\mathbf{i}<\mathbf{k},\boldsymbol\theta>}
\qquad \mbox{ such that } \qquad
\sum\limits_{\v{k}\in \mathbb{Z}^d} \left|f_{\v{k}}\right|<\infty.
\end{eqnarray*}
\end{definition}
We determine the sequence of symbols associated to
$\{\mathbf{\mathcal{M}}^{(m)}_{(\alpha,\beta),N}\}_{_{N\in
\mathbb{N}}}$ as a corollary of the following proposition.
\begin{pro}\label{pro2}
Let $\gamma\in(1,2)$ and let $A^{\gamma}_{\tilde{N}}$ be defined as in \eqref{eq:AN}. Then the symbol associated to the matrix-sequence
$\{A^{\gamma}_{\tilde{N}}\}_{_{\tilde{N}\in \mathbb{N}}}$
belongs to the Wiener class and its formal expression is given by
\begin{equation}\label{eq:fgamma}
f_{\gamma}(\xi)=-\sum\limits_{k=-1}^{\infty}w^{(\gamma)}_{k+1}{\rm
e}^{\mathbf{i}k\xi}=-\bigg[\frac{2-\gamma(1-{\rm
e}^{-\mathbf{i}\xi})}{2}\bigg]\bigg(1+{\rm
e}^{\mathbf{i}(\xi+\pi)}\bigg)^{\gamma}.
\end{equation}
\end{pro}
\begin{proof}
Let us observe that
$A^{\gamma}_{\tilde{N}}=[-w^{(\gamma)}_{i-j+1}]_{i,j=1}^{\tilde{N}}$,
with $w^{(\gamma)}_k=0$ for $k<0$, and let us define the function
\begin{equation}\label{eq:f_gamma}
f_{\gamma}(\xi)=-\sum\limits_{k=-1}^{\infty}w^{(\gamma)}_{k+1}{\rm
e}^{\mathbf{i}k\xi}.
\end{equation}
Then, by Definition \ref{def:toep}, it holds that $A^\gamma_{\tilde{N}}=T_{\tilde{N}}(f_\gamma)$. When  $\gamma\in(1,2)$, it is easy to
see that $f_{\gamma}(\xi)$ lies in the Wiener class. In detail, from
Proposition \ref{ww} we know that
$w_1^{(\gamma)}=\frac{2-\gamma-\gamma^2}{2}<0$,
$w_2^{(\gamma)}=\frac{\gamma(\gamma^2+\gamma-4)}{4}$,
$w_k^{(\gamma)}>0$ for $k\geq 0$ and $k\neq 1,2$,  and
$w_k^{(\gamma)}=0$ for $k<0$. Then
\begin{eqnarray*}
\sum\limits_{k=-1}^{\infty}|w^{(\gamma)}_{k+1}|=\sum\limits_{\substack{k=-1\\k\neq
0,1}}^{\infty}w^{(\gamma)}_{k+1}+|w^{(\gamma)}_{1}|+|w^{(\gamma)}_{2}|=\sum\limits_{\substack{k=-1\\k\neq
0,1}}^{\infty}w^{(\gamma)}_{k+1}+\frac{\gamma^2+\gamma-2}{2}+\frac{|\gamma(\gamma^2+\gamma-4)|}{4}.
\end{eqnarray*}
Again from Proposition \ref{ww}, we deduce that
\begin{eqnarray*}
\sum\limits_{k=0}^{\infty}w^{(\gamma)}_{k}=0\Leftrightarrow\sum\limits_{\substack{k=-1\\k\neq
0,1}}^{\infty}w^{(\gamma)}_{k+1}=-w^{(\gamma)}_{1}-w^{(\gamma)}_{2}=-\frac{2-\gamma-\gamma^2}{2}-\frac{\gamma(\gamma^2+\gamma-4)}{4}
=\frac{\gamma^2+\gamma-2}{2}-\frac{\gamma(\gamma^2+\gamma-4)}{4},
\end{eqnarray*}
that is
\begin{eqnarray*}
\sum\limits_{k=-1}^{\infty}|w^{(\gamma)}_{k+1}|=\gamma^2+\gamma-2+\frac{\gamma}{4}\bigg[|\gamma^2+\gamma-4|-(\gamma^2+\gamma-4)\bigg].
\end{eqnarray*}
The righthand side of the previous relation is a positive constant for $\gamma\in(1,2)$, then
we can conclude that $f_{\gamma}(\xi)$ belongs to the Wiener class. An explicit formula for the symbol
$f_{\gamma}(\xi)$ can be obtained combining the definition of $w^{(\gamma)}_{k}$ given in (\ref{w}) with the one of $g_k^{(\alpha)}$ in \eqref{eq:gk} as follows
\begin{eqnarray*}
f_{\gamma}(\xi)&=&-\sum\limits_{k=0}^{\infty}w^{(\gamma)}_{k}{\rm
e}^{\mathbf{i}(k-1)\xi}=-\frac{\gamma}{2}g^{(\gamma)}_{0}{\rm
e}^{-\mathbf{i}\xi}-\sum\limits_{k=1}^{\infty}\bigg(\frac{\gamma}{2}g^{(\gamma)}_{k}+\frac{2-\gamma}{2}g^{(\gamma)}_{k-1}\bigg){\rm
e}^{\mathbf{i}(k-1)\xi}\\
&=&-\frac{\gamma}{2}\sum\limits_{k=0}^{\infty}(-1)^k
\binom{\gamma}{k}{\rm
e}^{\mathbf{i}(k-1)\xi}-\frac{2-\gamma}{2}\sum\limits_{k=0}^{\infty}(-1)^k
\binom{\gamma}{k}{\rm e}^{\mathbf{i}k\xi}=-\frac{\gamma}{2}{\rm
e}^{-\mathbf{i}\xi}\sum\limits_{k=0}^{\infty}\binom{\gamma}{k}{\rm
e}^{\mathbf{i}k(\xi+\pi)}-\frac{2-\gamma}{2}\sum\limits_{k=0}^{\infty}
\binom{\gamma}{k}{\rm e}^{\mathbf{i}k(\xi+\pi)}.
\end{eqnarray*}
Applying the well-known binomial series
\begin{eqnarray*}
(1+z)^{\gamma}=\sum\limits_{k=0}^{\infty}
\binom{\gamma}{k}z^k,~~~~~z\in\mathbb{C},~~~|z|\leq 1,~~~\gamma>0,
\end{eqnarray*}
with $z={\rm e}^{\mathbf{i}(\xi+\pi)}$, the thesis follows.
\end{proof}
\begin{cor}
Let $\alpha,\beta\in(1,2)$ and let us assume that
$d_+(x,y,t)=d^+>0,~d_-(x,y,t)=d^->0,~e_+(x,y,t)=e^+>0,~e_-(x,y,t)=e^->0$.
Then the matrix
$\mathbf{\mathcal{M}}^{(m)}_{(\alpha,\beta),N}$
defined as in \eqref{eq:M} is the BTTB matrix
$\mathbf{\mathcal{M}}^{(m)}_{(\alpha,\beta),N}=T_N^{(2)}(\varphi_{(\alpha,\beta)})$
with
\begin{equation}\label{eq:phi}
\varphi_{(\alpha,\beta)}(\theta_1,\theta_2)=
\frac{1}{r}+d^+f_{\alpha}(\theta_1)+d^-f_{\alpha}(-\theta_1)+\frac{s}{r}e^+f_{\beta}(\theta_2)+\frac{s}{r}e^-f_{\beta}(-\theta_2),
\end{equation}
where $r=\frac{\triangle t}{2h_x^{\alpha}}$ and $s=\frac{\triangle t}{2h_y^{\beta}}$.
\end{cor}
\begin{proof}
According to Definition \ref{def:toep} and thanks to Proposition \ref{pro2}, the BTTB terms of $A_x^{(m)}$ and $A_y^{(m)}$ in equation \eqref{AxAy} can be written as
\begin{align*}
I_{n_2}\otimes A^{\alpha}_{n_1}& = T_{n_2}(1)\otimes T_{n_1}(f_{\alpha})  =  T_{N}^{(2)}(\tilde{f}_{\alpha}), \\
A^{\beta}_{n_2} \otimes I_{n_1}& = T_{n_2}(f_{\beta}) \otimes T_{n_1}(1) = T_{N}^{(2)}(\tilde{f}_{\beta}).
\end{align*}
where
$ \tilde{f}_{\alpha}(\theta_1,\theta_2) =  f_{\alpha}(\theta_1)$ and
$\tilde{f}_{\beta}(\theta_1,\theta_2)=f_{\beta}(\theta_2)$.

Finally, recalling that for a real function $f$ it holds $T_N^{(d)}(f(\boldsymbol\theta))^T=T_N^{(d)}(f(-\boldsymbol\theta))$ and replacing equations \eqref{AxAy} in \eqref{eq:M}, we obtain $\mathbf{\mathcal{M}}^{(m)}_{(\alpha,\beta),N}=T_N^{(2)}(\varphi_{(\alpha,\beta)})$ with $\varphi_{(\alpha,\beta)}$ defined as in \eqref{eq:phi}.
\end{proof}

Now we focus our attention on the spectral distribution of
$\{\mathbf{\mathcal{M}}^{(m)}_{(\alpha,\beta),N}\}_{_{N\in
\mathbb{N}}}$, under the further
assumption that the diffusion coefficients are equal on both sides. By this
hypothesis, $\{\mathbf{\mathcal{M}}^{(m)}_{(\alpha,\beta),N}\}_{_{N\in
\mathbb{N}}}$ is a
sequence of symmetric BTTB matrices. Let us start with the
definition of the spectral distribution in the sense of the
eigenvalues and of the singular values.
\begin{definition}
Let $f:G\to\mathbb{C}$ be a measurable function, defined on a
measurable set $G\subset\mathbb R^k$ with $k\ge 1$,
$0<m_k(G)<\infty$. Let  $\mathcal C_0(\mathbb K)$ be the set of
continuous functions with compact support over $\mathbb K\in
\{\mathbb C, \mathbb R_0^+\}$ and let $\{A_N\}$ be a sequence of
matrices of size $N$ with eigenvalues $\lambda_j(A_N)$,
$j=1,\ldots,N$ and singular values $\sigma_j(A_N)$, $j=1,\ldots,N$.
\begin{itemize}
    \item $\{A_N\}$ is {\em distributed as the pair $(f,G)$ in the sense of the eigenvalues,} in symbols $\{A_N\}\sim_\lambda(f,G)$, if the following limit relation holds for all $F\in\mathcal C_0(\mathbb C)$
\begin{align*}
  \lim_{N\to\infty}\frac{1}{N}\sum_{j=1}^{N}F(\lambda_j(A_N))=
  \frac1{m_k(G)}\int_G F(f(t)) dt.
\end{align*}
    \item $\{A_N\}$ is {\em distributed as the pair $(f,G)$ in the sense of the singular values,} in symbols $\{A_N\}\sim_\sigma(f,G)$, if the following limit relation holds for all $F\in\mathcal C_0(\mathbb R_0^+)$
\begin{align*}
  \lim_{N\to\infty}\frac{1}{N}\sum_{j=1}^{N}F(\sigma_j(A_N))=
  \frac1{m_k(G)}\int_G F(|f(t)|) dt.
\end{align*}
\end{itemize}
\end{definition}

For Hermitian $d$-level Toeplitz matrix-sequences, the following theorem due to
Szeg\"{o} and Tilli holds (see \cite{14,Tilli}).
\begin{theorem}\label{theorem1}
Let $f\in L^1([-\pi,\pi]^d)$ be a real-valued function, then
\begin{eqnarray*}
\left\{T^{(d)}_N(f)\right\}_{_{N\in\mathbb{N}}}\sim_{\lambda}\bigg(f,[-\pi,\pi]^d\bigg).
\end{eqnarray*}
\end{theorem}

Now, we recall a property of the spectral norm of $d$-level Toeplitz matrices and we state a relevant theorem contained in \cite{13}. Given a square matrix $A$ of order $N$, we denote its spectral norm by $\|A\|=\sigma_{1}(A)$ and its trace norm by $\|A\|_1=\sum\limits_{i=1}^{N} \sigma_i(A)$. For the spectral norm of a $d$-level Toeplitz sequence $\{T^{(d)}_N(f)\}_{_{N\in \mathbb{N}}}$ generated by $f$ it holds (see Corollary 3.5 in \cite{23})
\begin{eqnarray}\label{10}
f\in L^{\infty}(-\pi,\pi]^d\Rightarrow \|T^{(d)}_N(f)\|\leq
\|f\|_{\infty},~~~~~\forall N\in \mathbb{N}.
\end{eqnarray}

%

\begin{theorem}\label{theorem2}
(Theorem 3.4 in \cite{13}) Let $\{A_N\}_{_{N\in \mathbb{N}}}$ be
a matrix-sequence with $A_N = B_N+C_N$ and $B_N$ Hermitian $\forall
N\in \mathbb{N}$. Assume that
\begin{itemize}
    \item $\{B_N\}_{_{N\in
\mathbb{N}}}\sim_{\lambda} (f,G)$,
    \item $\|B_N\|,~\|C_N\|$ are bounded by a constant
independent of $N$,
    \item $\|C_N\|_1=o(N)$.
\end{itemize}
Then $\{A_N\}_{_{N\in \mathbb{N}}}\sim_{\lambda} (f,G)$.
\end{theorem}
The following proposition concerns the eigenvalue distribution
of the coefficient matrix-sequence
$\{\mathbf{\mathcal{M}}^{(m)}_{(\alpha,\beta),N}\}_{_{N\in
\mathbb{N}}}$ when the diffusion coefficients are constant and equal.
\begin{pro}\label{PaPb}
Let us assume that $d_{\pm}(x,y,t)=d>0,~e_{\pm}(x,y,t)=e>0$, that
$\frac{1}{r}=o(1)$, and
$\frac{s}{r}=\frac{h_x^{\alpha}}{h_y^{\beta}}=\mathcal{O}(1)$.
Let $f_{\gamma}$ be defined as in \eqref{eq:fgamma} and define
\[
q_{\gamma}(\xi)=f_{\gamma}(\xi)+f_{\gamma}(-\xi)=f_{\gamma}(\xi)+\overline{f_{\gamma}(\xi)}.
\]
Given the matrix-sequence
$\{\mathbf{\mathcal{M}}^{(m)}_{(\alpha,\beta),N}\}_{_{N\in
\mathbb{N}}}$, we have
\begin{eqnarray*}
\left\{\mathbf{\mathcal{M}}^{(m)}_{(\alpha,\beta),N}\right\}_{_{N\in
\mathbb{N}}}\sim_{\lambda}\left(\mathbf{\mathcal{F}}_{(\alpha,\beta)}(\theta_1,\theta_2),[-\pi,\pi]^2\right),
\end{eqnarray*}
where
\begin{equation}\label{P}
\mathbf{\mathcal{F}}_{(\alpha,\beta)}(\theta_1,\theta_2)=d\, q_{\alpha}(\theta_1)+e\frac{s}{r}\, q_{\beta}(\theta_2),
\end{equation}
is a real-valued continuous function and it is nonnegative for $\alpha,\beta \in (1,2)$.
\end{pro}
\begin{proof}
Since the diffusion coefficients $d_{\pm}(x,y,t)$ and
$e_{\pm}(x,y,t)$ are constant and equal to real positive numbers $d$
and $e$, respectively, the matrices of the sequence
$\{A^{(m)}_x+\frac{s}{r}A^{(m)}_y\}_{_{N\in \mathbb{N}}}$
(see relation (\ref{AxAy})) are symmetric. The function
\begin{eqnarray*}
&&\mathbf{\mathcal{F}}_{(\alpha,\beta)}(\theta_1,\theta_2)=d\, q_{\alpha}(\theta_1)+e\frac{s}{r}\, q_{\beta}(\theta_2),\\
&&q_{\gamma}(\xi)=f_{\gamma}(\xi)+f_{\gamma}(-\xi)=f_{\gamma}(\xi)+\overline{f_{\gamma}(\xi)},
\end{eqnarray*}
belongs to the Wiener algebra since $f_{\gamma}(\xi)$ itself is in
the same algebra (see Proposition \ref{pro2}). Furthermore, from its
expression, it also follows that $q_{\gamma}(\xi)$ is
real-valued and globally continuous. Hence,
$\mathbf{\mathcal{F}}_{(\alpha,\beta)}(\theta_1,\theta_2)$ is
real-valued and globally continuous.
Similarly, the nonnegativity of $\mathbf{\mathcal{F}}_{(\alpha,\beta)}(\theta_1,\theta_2)$  follows from the nonnegativity of
$q_{\gamma}(\xi)$ which in turn can be easily derived from the expression of $f_{\gamma}(\xi)$ in \eqref{eq:fgamma}  for $\gamma\in(1,2)$ and $\xi\in\{\theta_1,\theta_2\}$.

From  Theorem \ref{theorem1} with $d=2$ and since
$\frac{s}{r}=\frac{h_x^{\alpha}}{h_y^{\beta}}=\mathcal{O}(1)$, it
follows that $\{A^{(m)}_x+\frac{s}{r}A^{(m)}_y\}_{_{N\in
\mathbb{N}}}\sim_{\lambda}(\mathbf{\mathcal{F}}_{(\alpha,\beta)}(\theta_1,\theta_2),[-\pi,\pi]^2)$.
Furthermore, using (\ref{10}) with $d=2$, we
have that
\begin{eqnarray*}
\left\|A^{(m)}_x+\frac{s}{r}A^{(m)}_y\right\|&\leq&
C\|\mathbf{\mathcal{F}}_{(\alpha,\beta)}(\theta_1,\theta_2)\|_{\infty}\\
&\leq&
C\left[d\|q_{\alpha}(\theta_1)\|_{\infty}+e\frac{s}{r}\|q_{\beta}(\theta_2)\|_{\infty}\right]\\
&\leq& C(d(\alpha-1)2^{\alpha+1}+e\frac{s}{r}(\beta-1)2^{\beta+1}),
\end{eqnarray*}
with $C$ independent of $N$. Moreover, under the hypothesis that
$\frac{1}{r}=o(1)$, the remaining term $\frac{1}{r}I_{N}$ is such
that $\|\frac{1}{r}I_{N}\|_1 = o(N)$ and
$\|\frac{1}{r}I_{N}\|=\frac{1}{r}<\tilde{C}$ for some constant
$\tilde{C}$ independent of $N$. By Theorem \ref{theorem2}, we
conclude that the distribution of
$\{\mathbf{\mathcal{M}}^{(m)}_{(\alpha,\beta),N}\}_{_{N\in
\mathbb{N}}}$ is decided only by
$\mathbf{\mathcal{F}}_{(\alpha,\beta)}(\theta_1,\theta_2)$.
\end{proof}
Let us recall that both $q_{\alpha}(\theta_1)$ and $q_{\beta}(\theta_2)$ are nonnegative functions. Moreover, as a straightforward consequence of Proposition 4 in \cite{M2}, it is easy to prove that $q_{\alpha}(\theta_1)$ and $q_{\beta}(\theta_2)$ have a zero at $0$ of order $\alpha$ and $\beta$, respectively. With this result in mind, in the following proposition we prove that the superior limit of $\frac{\mathbf{\mathcal{F}}_{(\alpha,\beta)}(\theta_1,\theta_2)}{|\boldsymbol\theta|^\gamma}$, with $\gamma=\min\{\alpha,\beta\}$, is bounded as $(\theta_1,\theta_2)\to(0,0)$. Such a proposition will be used in Section \ref{sec:mgm} for proving the constant converge rate of the two-grid and the V-cycle.
\begin{pro}\label{pro4}
Let $\alpha,\beta\in(1,2)$, with $\alpha\ne\beta$, and let $\gamma=\min\{\alpha,\beta\}$, then there exist two real constants $C_1,C_2>0$ such that
\begin{equation*}
C_1<\limsup\limits_{\boldsymbol\theta\to\boldsymbol0}\frac{\mathbf{\mathcal{F}}_{(\alpha,\beta)}(\theta_1,\theta_2)}{|\boldsymbol\theta|^{\gamma}}<C_2,
\end{equation*}
where $\mathbf{0}=(0,0)$ and $\boldsymbol\theta=(\theta_1,\theta_2)$.
\end{pro}
\begin{proof}
Let us rewrite $1+{\rm e}^{{\bf i}(\theta_k+\pi)}$ and $1-{\rm e}^{-{\bf i}\theta_k}$ in polar form
\[
1+{\rm e}^{{\bf i}(\theta_k+\pi)}= \sqrt{2-2\cos \theta_k}\ {\rm e}^{{\bf i}\phi_k},
\qquad
1-{\rm e}^{-{\bf i}\theta_k}= \sqrt{2-2\cos \theta_k}\ {\rm e}^{-{\bf i}\phi_k},
\]
where
\begin{eqnarray}\label{limtan}
\phi_k=\left\{
\begin{array}{lc}
\arctan\left(\frac{-\sin \theta_k}{1-\cos \theta_k}\right), & \theta_k\ne0 \\
\lim_{\theta_k\rightarrow 0^{+}}\arctan\left(\frac{-\sin \theta_k}{1-\cos \theta_k}\right)=-\frac{\pi}{2}, & \theta_k=0
\end{array}
\right.
\end{eqnarray}
for $k=1,2$. According to the polar form we have that
\begin{eqnarray}\label{lim}\nonumber
&&\limsup\limits_{\boldsymbol\theta\to\boldsymbol0}\frac{\mathbf{\mathcal{F}}_{(\alpha,\beta)}(\theta_1,\theta_2)}{|\boldsymbol\theta|^{\gamma}}\le
\limsup\limits_{\boldsymbol\theta\to\boldsymbol0}\frac{d\,
q_{\alpha}(\theta_1)}{(\theta_1^2+\theta_2^2)^{\frac{\gamma}{2}}}+
\limsup\limits_{\boldsymbol\theta\to\boldsymbol0}\frac{e\frac{s}{r}\,
q_{\beta}(\theta_2)}{(\theta_1^2+\theta_2^2)^{\frac{\gamma}{2}}}\\\nonumber
&=&-2d\limsup\limits_{\boldsymbol\theta\to\boldsymbol0}\frac{(2-2\cos\theta_1)^\frac{\alpha}{2}\cos(\alpha\phi_1)-\frac{\alpha}{2}(2-2\cos\theta_1)^\frac{\alpha+1}{2}\cos((\alpha-1)\phi_1)}
{(\theta_1^2+\theta_2^2)^{\frac{\gamma}{2}}}\\\nonumber
&&-2e\frac{s}{r}\limsup\limits_{\boldsymbol\theta\to\boldsymbol0}\frac{(2-2\cos\theta_2)^\frac{\beta}{2}\cos(\beta\phi_2)-\frac{\beta}{2}(2-2\cos\theta_2)^\frac{\beta+1}{2}\cos((\beta-1)\phi_2)}
{(\theta_1^2+\theta_2^2)^{\frac{\gamma}{2}}}\\\nonumber
&\le&-2d\limsup\limits_{\boldsymbol\theta\to\boldsymbol0}\frac{(2-2\cos\theta_1)^\frac{\alpha}{2}}{(\theta_1^2+\theta_2^2)^{\frac{\gamma}{2}}}\cos(\alpha\phi_1)
+d\,\alpha\limsup\limits_{\boldsymbol\theta\to\boldsymbol0}\frac{(2-2\cos\theta_1)^\frac{\alpha+1}{2}}{(\theta_1^2+\theta_2^2)^{\frac{\gamma}{2}}}\cos((\alpha-1)\phi_1)\\
&&-2e\frac{s}{r}\limsup\limits_{\boldsymbol\theta\to\boldsymbol0}\frac{(2-2\cos\theta_2)^\frac{\beta}{2}}{(\theta_1^2+\theta_2^2)^{\frac{\gamma}{2}}}\cos(\beta\phi_2)
+e\frac{s}{r}\,\beta\limsup\limits_{\boldsymbol\theta\to\boldsymbol0}\frac{(2-2\cos\theta_2)^\frac{\beta+1}{2}}{(\theta_1^2+\theta_2^2)^{\frac{\gamma}{2}}}\cos((\beta-1)\phi_2).
\end{eqnarray}
Moreover, we have that
\begin{eqnarray}\label{aa1}
\limsup\limits_{\boldsymbol\theta\to\boldsymbol0}\frac{(2-2\cos\theta_1)^\frac{\alpha}{2}}{(\theta_1^2+\theta_2^2)^{\frac{\gamma}{2}}}\leq
\limsup\limits_{\theta_1\to0}\frac{(2-2\cos\theta_1)^\frac{\alpha}{2}}{(\theta_1^2)^{\frac{\gamma}{2}}},
\end{eqnarray}
and
\begin{eqnarray}\label{bb1}
\limsup\limits_{\boldsymbol\theta\to\boldsymbol0}\frac{(2-2\cos\theta_2)^\frac{\beta}{2}}{(\theta_1^2+\theta_2^2)^{\frac{\gamma}{2}}}\leq
\limsup\limits_{\theta_2\to0}\frac{(2-2\cos\theta_2)^\frac{\beta}{2}}{(\theta_2^2)^{\frac{\gamma}{2}}}.
\end{eqnarray}
Note that if $\alpha-\gamma\geq0$ and
$\beta-\gamma\geq0$, i.e., $\gamma=\min\{\alpha,\beta\}$, then both limits
(\ref{aa1}) and (\ref{bb1}) are finite. Moreover, since
$$\lim\limits_{\theta_k\to0}\cos(\eta\phi_k)=\cos\left(\eta\frac{\pi}{2}\right)<0,\quad \eta\in(1,2), \quad k=1,2,$$
from relation (\ref{lim}) it holds that
\begin{enumerate}
\item if $\gamma=\alpha$, then
\begin{eqnarray*}
\limsup\limits_{\boldsymbol\theta\to\boldsymbol0}\frac{\mathbf{\mathcal{F}}_{(\alpha,\beta)}(\theta_1,\theta_2)}{|\boldsymbol\theta|^{\gamma}}
\leq -2d\cos\bigg(\alpha\frac{\pi}{2}\bigg)\in(0,2d);
\end{eqnarray*}

\item if $\gamma=\beta$, then
\begin{eqnarray*}
\limsup\limits_{\boldsymbol\theta\to\boldsymbol0}\frac{\mathbf{\mathcal{F}}_{(\alpha,\beta)}(\theta_1,\theta_2)}{|\boldsymbol\theta|^{\gamma}}
\leq -2e\frac{s}{r}\cos\bigg(\beta\frac{\pi}{2}\bigg)\in\left(0,2e\frac{s}{r}\right).
\end{eqnarray*}
\end{enumerate}
It is to convince the reader that $\lim\limits_{\boldsymbol\theta\to\boldsymbol0}
\frac{\mathbf{\mathcal{F}}_{(\alpha,\beta)}(\theta_1,\theta_2)}{|\boldsymbol\theta|^{\gamma}}$ does not exists. Indeed, if w.l.o.g. we fix $\gamma=\alpha<\beta$, then along the lines $\theta_1=0$ and $\theta_2=0$ we get
\begin{eqnarray*}
\begin{array}{lllll}
\lim\limits_{\theta_1\to0}\frac{\mathbf{\mathcal{F}}_{(\alpha,\beta)}(\theta_1,0)}{|\theta_1|^{\alpha}}&=&
\lim\limits_{\theta_1\to0}\frac{d\,q_\alpha(\theta_1)}{|\theta_1|^{\alpha}}&\overset{(\triangle)}{=}&C,\\
\lim\limits_{\theta_2\to0}\frac{\mathbf{\mathcal{F}}_{(\alpha,\beta)}(0,\theta_2)}{|\theta_2|^{\alpha}}&=&
\lim\limits_{\theta_2\to0}\frac{e\frac{s}{r}\,q_\beta(\theta_2)}{|\theta_2|^{\alpha}}&\overset{(\triangle)}{=}&0,
\end{array}
\end{eqnarray*}
where $C$ is a positive constant. The equalities $(\triangle)$ are due to the fact that $q_{\alpha}(\theta_1)$ and $q_{\beta}(\theta_2)$ have a zero at $0$ of order $\alpha$ and $\beta$, respectively, with $\alpha<\beta$ by hypothesis. Therefore, it yields that
\begin{equation*}
\liminf\limits_{\boldsymbol\theta\to\boldsymbol0}\frac{\mathbf{\mathcal{F}}_{(\alpha,\beta)}(\theta_1,\theta_2)}{|\boldsymbol\theta|^{\gamma}}
<\limsup\limits_{\boldsymbol\theta\to\boldsymbol0}\frac{\mathbf{\mathcal{F}}_{(\alpha,\beta)}(\theta_1,\theta_2)}{|\boldsymbol\theta|^{\gamma}},
\end{equation*}
and this, observing that $\liminf\limits_{\boldsymbol\theta\to\boldsymbol0}\frac{\mathbf{\mathcal{F}}_{(\alpha,\beta)}(\theta_1,\theta_2)}{|\boldsymbol\theta|^{\gamma}}
$ is nonnegative, completes the proof.
\end{proof}
\subsection{Nonconstant diffusion coefficients case}
Now we focus on the symbol associated to
$\{\mathbf{\mathcal{M}}^{(m)}_{(\alpha,\beta),N}\}_{_{N\in
\mathbb{N}}}$ and on its spectral distribution, when $d_+(x,y,t)$,
$d_-(x,y,t)$, $e_+(x,y,t)$ and $e_-(x,y,t)$ are nonconstant functions. For
this purpose we need the notion of GLT class and the related theory, starting from the pioneering
work by Tilli \cite{28} and widely generalized in \cite{2525}. In short, the GLT class is an algebra virtually containing any sequence of matrices coming from ``reasonable'' approximations by local discretization methods (finite differences, finite elements, isogeometric analysis, etc) of partial differential equations (see \cite{11}), containing the multilevel Toeplitz sequences with Lebesgue integrable generating functions. The formal definition is rather technical and involves a heavy notation: therefore we just give and briefly discuss the notion in two dimensions, which is the case of interest in our setting, and we report few properties of the GLT class \cite{11} , which are sufficient for studying the spectral features of the matrices $\left\{{\cal M}^{(m)}_{(\alpha,\beta),N}\right\}_{N\in\mathbb{N}}$


Since a GLT sequence is a sequence of matrices obtained from a combination of some algebraic operations on multilevel Toeplitz matrices and diagonal sampling matrices, we need the following definition.

\begin{definition}
Given a Riemann-integrable function $a:[0,1]^2 \rightarrow \mathbb{C}$, by \emph{diagonal sampling matrix} of order $N=n_1n_2$ we mean
$D_N(a) = {\rm diag}_{_{j_1=1,\dots,n_1\atop j_2=1,\dots,n_2}}
a\bigg(\frac{j_1}{n_1},\frac{j_2}{n_2}\bigg)$.
\end{definition}

Throughout, we use the following notation
$$\{A_N\}_{N\in\mathbb{N}}\sim_{GLT} { \kappa(\mathbf{x},\boldsymbol\theta)},$$
to say that the sequence $\{A_N\}_{N\in\mathbb{N}}$ is a GLT sequence with symbol $\kappa(\mathbf{x},\boldsymbol\theta)$.

Here we report four main features of the GLT class in two dimensions.

\begin{description}
\item[GLT1] Let $\{A_N\}_{N\in\mathbb{N}}\sim_{GLT}\kappa(\mathbf{x},\boldsymbol\theta)$ with $\kappa:G\rightarrow\mathbb{C}$, $G=[0,1]^2\times[-\pi,\pi]^2$,
then $\{A_N\}_{N\in\mathbb{N}}\sim_{\sigma}(\kappa,G)$. If the
matrices $A_N$ are Hermitian, then it holds also
$\{A_N\}_{N\in\mathbb{N}}\sim_{\lambda}(\kappa,G)$.

\item[GLT2] The set of GLT sequences form a $\ast$-algebra, i.e., it is closed
under linear combinations, products, inversion (whenever the symbol
vanishes, at most, in a set of zero Lebesgue measure), conjugation:
hence, the sequence obtained via algebraic operations on a finite
set of input GLT sequences is still a GLT sequence and its symbol is
obtained by following the same algebraic manipulations on the
corresponding symbols of the input GLT sequences.

\item[GLT3] Every BTTB sequence $\{T^{(2)}_N(f)\}_{N\in\mathbb{N}}$ generated by an $L^1([-\pi,\pi]^2)$ function $f(\boldsymbol\theta)$ is such that $\{T^{(2)}_N(f)\}_{N\in\mathbb{N}}\sim_{GLT}f(\boldsymbol\theta)$, with the specifications
reported in item $\mathbf{[GLT1]}$. Every diagonal sampling sequence
$\{D_N(a)\}_{N\in\mathbb{N}}$, where
$a(\mathbf{x})$ is a Riemann integrable function, is such that
$\{D_N(a)\}_{N\in\mathbb{N}}\sim_{GLT}a(\mathbf{x})$.

\item[GLT4] Let $\{A_N\}_{N\in\mathbb{N}}\sim_{\sigma}(0,G)$, $G=[0,1]^2\times[-\pi,\pi]^2$, then $\{A_N\}_{N\in\mathbb{N}}\sim_{GLT}0$.
\end{description}

\begin{pro}
Let us assume that $\frac{1}{r}=o(1)$ and
$\frac{s}{r}=\frac{h_x^{\alpha}}{h_y^{\beta}}=\mathcal{O}(1)$ and
that, fixed the instant of time $t_m$, $d_+(x,y):=d_+(x,y,t_m)$,
$d_-(x,y):=d_-(x,y,t_m)$, $e_+(x,y):=e_+(x,y,t_m)$ and
$e_-(x,y):=e_-(x,y,t_m)$ are Riemann integrable over
$[a_1,b_1]\times[a_2,b_2]$. For the matrix
$\mathbf{\mathcal{M}}^{(m)}_{(\alpha,\beta),N}$ it holds
\begin{eqnarray*}
\left\{\mathbf{\mathcal{M}}^{(m)}_{(\alpha,\beta),N}\right\}_{_{N\in
\mathbb{N}}}\sim_{GLT}
\hat{h}_{(\alpha,\beta)}(\hat{\mathbf{x}},\boldsymbol\theta),~~~~~~~~~~\hat{\mathbf{x}}=(\hat{x},\hat{y}),~~\boldsymbol\theta=(\theta_1,\theta_2),
\end{eqnarray*}
with
\begin{eqnarray}\label{h}\nonumber
&&\hat{h}_{(\alpha,\beta)}(\hat{\mathbf{x}},\boldsymbol\theta)=h_{(\alpha,\beta)}(a_1+(b_1-a_1)\hat{x},a_2+(b_2-a_2)\hat{y},\boldsymbol\theta),\\
&&h_{(\alpha,\beta)}(\mathbf{x},\boldsymbol\theta)=g_{\alpha}(\mathbf{x},\theta_1)+\frac{s}{r}g_{\beta}(\mathbf{x},\theta_2),~~~~~\mathbf{x}=(x,y),
\end{eqnarray}
where
\begin{align}\label{g_alphabeta}
\begin{split}
g_{\alpha}(\mathbf{x},\theta_1)&=d_+(x,y)f_{\alpha}(\theta_1)+d_-(x,y)f_{\alpha}(-\theta_1),\\
g_{\beta}(\mathbf{x},\theta_2)&=e_+(x,y)f_{\beta}(\theta_2)+e_-(x,y)f_{\beta}(-\theta_2),
\end{split}
\end{align}
and
$(\hat{\mathbf{x}},\boldsymbol\theta)\in[0,1]^2\times[-\pi,\pi]^2,~(\mathbf{x},\boldsymbol\theta)\in[a_1,b_1]\times[a_2,b_2]\times[-\pi,\pi]^2$.
Furthermore,
\begin{eqnarray*}
\left\{\mathbf{\mathcal{M}}^{(m)}_{(\alpha,\beta),N}\right\}_{_{N\in
\mathbb{N}}}\sim_{\sigma}
\left(h_{(\alpha,\beta)}(\mathbf{x},\boldsymbol\theta),[a_1,b_1]\times[a_2,b_2]\times[-\pi,\pi]^2\right),
\end{eqnarray*}
and whenever $d_+(x,y)=d_-(x,y)=e_+(x,y)=e_-(x,y)$, we also have
\begin{eqnarray*}
\left\{\mathbf{\mathcal{M}}^{(m)}_{(\alpha,\beta),N}\right\}_{_{N\in
\mathbb{N}}}\sim_{\lambda}
\left(h_{(\alpha,\beta)}(\mathbf{x},\boldsymbol\theta),[a_1,b_1]\times[a_2,b_2]\times[-\pi,\pi]^2\right),
\end{eqnarray*}
with $h_{(\alpha,\beta)}(\mathbf{x},\boldsymbol\theta)$
real-valued and indeed all the matrices
$\mathbf{\mathcal{M}}^{(m)}_{(\alpha,\beta),N}$ have only real
eigenvalues.
\end{pro}
\begin{proof}
Let us observe that, fixed the instant of time $t_m$, the diagonal
elements of the matrices $D^{(m)}_{\pm}$ and $E^{(m)}_{\pm}$ are a
uniform sampling of the functions $d_{\pm}(\mathbf{x})$, and $e_{\pm}(\mathbf{x})$, respectively, with
$\mathbf{x}=(x,y)\in[a_1,b_1]\times[a_2,b_2]$. Therefore, thanks to [\textbf{GLT3}] and to the Riemann integrability of the diffusion coefficients it yields
\begin{eqnarray*}
&&\left\{D^{(m)}_{\pm}\right\}_{_{N\in \mathbb{N}}}\sim_{GLT}
\hat{d}_{\pm}(\hat{\mathbf{x}})=d_{\pm}(a_1+(b_1-a_1)\hat{x},a_2+(b_2-a_2)\hat{y}),~~~~~~~\hat{\mathbf{x}}=(\hat{x},\hat{y})\in[0,1]^2,\\
&&\left\{E^{(m)}_{\pm}\right\}_{_{N\in \mathbb{N}}}\sim_{GLT}
\hat{e}_{\pm}(\hat{\mathbf{x}})=e_{\pm}(a_1+(b_1-a_1)\hat{x},a_2+(b_2-a_2)\hat{y}),~~~~~~~\hat{\mathbf{x}}=(\hat{x},\hat{y})\in[0,1]^2.
\end{eqnarray*}
Since the GLT class is stable under linear combinations and
products [\textbf{GLT2}] and since BTTB sequences with $L^1([-\pi,\pi]^2)$ symbols lie in the
GLT class [\textbf{GLT3}], it is immediate to see that, under the hypothesis that
$\frac{s}{r}=\frac{h_x^{\alpha}}{h_y^{\beta}}=\mathcal{O}(1)$, the matrix-sequence
$\{A^{(m)}_x+\frac{s}{r}A^{(m)}_y\}_{_{N\in \mathbb{N}}}$ is still a member  of the GLT class and
\begin{eqnarray}\label{eq:distr}
\{A^{(m)}_x+\frac{s}{r}A^{(m)}_y\}_{_{N\in \mathbb{N}}}\sim_{GLT}\hat{h}_{(\alpha,\beta)}(\hat{\mathbf{x}},\boldsymbol\theta)=g_{\alpha}(\hat{\mathbf{x}},\theta_1)
+\frac{s}{r}g_{\beta}(\hat{\mathbf{x}},\theta_2),
\end{eqnarray}
where $g_\alpha$, $g_\beta$ are defined as in \eqref{g_alphabeta} and $(\hat{\mathbf{x}},\boldsymbol\theta)\in[0,1]^2\times[-\pi,\pi]^2$. Moreover, the sequence $\{\frac{1}{r}I_{N}\}_{_{N\in \mathbb{N}}}$ (under the hypothesis that $\frac{1}{r}=o(1)$) is a GLT sequence
with zero symbol, item [\textbf{GLT4}]. This together with [\textbf{GLT2}] and \eqref{eq:distr} implies that
$\{\mathbf{\mathcal{M}}^{(m)}_{(\alpha,\beta),N}\}_{_{N\in
\mathbb{N}}}\sim_{GLT}
\hat{h}_{(\alpha,\beta)}(\hat{\mathbf{x}},\boldsymbol\theta)$. Then by [\textbf{GLT1}] we can conclude
$\{\mathbf{\mathcal{M}}^{(m)}_{(\alpha,\beta),N}\}_{_{N\in
\mathbb{N}}}\sim_{\sigma}
(\hat{h}_{(\alpha,\beta)}(\hat{\mathbf{x}},\boldsymbol\theta),[0,1]^2\times[-\pi,\pi]^2)$
and hence
$\{\mathbf{\mathcal{M}}^{(m)}_{(\alpha,\beta),N}\}_{_{N\in
\mathbb{N}}}\sim_{\sigma}
(h_{(\alpha,\beta)}(\mathbf{x},\boldsymbol\theta),[a_1,b_1]\times[a_2,b_2]\times[-\pi,\pi]^2)$,
after an affine change of variable.

Now, by exploiting Proposition \ref{pro2} and Proposition
\ref{PaPb}, since
$\mathbf{\mathcal{F}}_{(\alpha,\beta)}(\boldsymbol\theta)$ is
real-valued, it is clear that if $d_+(\mathbf{x})=d_-(\mathbf{x})=e_+(\mathbf{x})=e_-(\mathbf{x})$ then $h_{(\alpha,\beta)}(\mathbf{x},\boldsymbol\theta)$ is real-valued. Furthermore, under the hypothesis $d_+(\mathbf{x})=d_-(\mathbf{x})=e_+(\mathbf{x})=e_-(\mathbf{x})$ we deduce that
$D^{(m)}_+=D^{(m)}_-=E^{(m)}_+=E^{(m)}_-$ is a positive definite diagonal
block matrix, and choosing $G$ as the positive definite square
root of $D^{(m)}_+$ , we find that
$G^{-1}\mathbf{\mathcal{M}}^{(m)}_{(\alpha,\beta),N}G$ is similar to
$\mathbf{\mathcal{M}}^{(m)}_{(\alpha,\beta),N}$ and real symmetric.
Therefore, all the eigenvalues of
$\mathbf{\mathcal{M}}^{(m)}_{(\alpha,\beta),N}$ are real and we
plainly have
$\{\mathbf{\mathcal{M}}^{(m)}_{(\alpha,\beta),N}\}_{_{N\in
\mathbb{N}}}\sim_{\lambda}
(h_{(\alpha,\beta)}(\mathbf{x},\boldsymbol\theta),[a_1,b_1]\times[a_2,b_2]\times[-\pi,\pi]^2)$,
by exploiting again the GLT machinery, as done before but in the
Hermitian setting.
\end{proof}
Now, let us assume that all the diffusion coefficients are uniformly bounded and positive. Under this hypothesis, the following proposition can be seen as an extension to the nonconstant coefficients case of the result for constant coefficients shown in Proposition \ref{pro4}.

\begin{pro}\label{pro6}
Let us assume that $\alpha\ne\beta$. Given $\mathbf{\mathcal{F}}_{(\alpha,\beta)}(\boldsymbol\theta)$ as in (\ref{P}) and $h_{(\alpha,\beta)}(\mathbf{x},\boldsymbol\theta)$ as in
(\ref{h}), the following limit relation holds
\begin{eqnarray*}
&&\limsup\limits_{\boldsymbol\theta\to\mathbf{0}}\frac{h_{(\alpha,\beta)}(\mathbf{x},\boldsymbol\theta)}{\mathbf{\mathcal{F}}_{(\alpha,\beta)}(\boldsymbol\theta)}=
\begin{cases}
\frac{d_+(x,y)+d_-(x,y)}{2d}-\mathbf{i}\tan(\alpha\frac{\pi}{2})\frac{d_+(x,y)-d_-(x,y)}{2d},\qquad
~~~~~~~\gamma=\alpha,
    \cr
\frac{e_+(x,y)+e_-(x,y)}{2e}-\mathbf{i}\tan(\beta\frac{\pi}{2})\frac{e_+(x,y)-e_-(x,y)}{2e},\qquad
~~~~~~~\gamma=\beta,
\end{cases}\\
\end{eqnarray*}
where $\gamma=\min\{\alpha,\beta\}$, $\mathbf{x}=(x,y)$,
$\mathbf{0}=(0,0)$ and $\boldsymbol\theta=(\theta_1,\theta_2)$.
\end{pro}
\begin{proof}
As in the proof of Proposition \ref{pro4}, we exploit the polar form
of $1+e^{\mathbf{i}(\theta_k+\pi)}$ and $1-e^{-\mathbf{i}\theta_k}$, for $k=1,2$, and rewrite
$h_{(\alpha,\beta)}(\mathbf{x},\boldsymbol\theta)$ and
$\mathbf{\mathcal{F}}_{(\alpha,\beta)}(\boldsymbol\theta)$ as
follows
\begin{align}\label{HHH}\nonumber
&\hspace{-0.5cm}
h_{(\alpha,\beta)}(\mathbf{x},\boldsymbol\theta)=\\\nonumber
\hspace{-3cm}
=&-d_+(x,y)\bigg((2-2\cos\theta_1)^{\frac{\alpha}{2}}e^{\mathbf{i}\alpha\phi_1}
-\frac{\alpha}{2}(2-2\cos\theta_1)^{\frac{\alpha+1}{2}}e^{\mathbf{i}(\alpha-1)\phi_1}\bigg)\\\nonumber
&-d_-(x,y)\bigg((2-2\cos\theta_1)^{\frac{\alpha}{2}}e^{-\mathbf{i}\alpha\phi_1}
-\frac{\alpha}{2}(2-2\cos\theta_1)^{\frac{\alpha+1}{2}}e^{-\mathbf{i}(\alpha-1)\phi_1}\bigg)\\\nonumber
&-\frac{s}{r}e_+(x,y)\bigg((2-2\cos\theta_2)^{\frac{\beta}{2}}e^{\mathbf{i}\beta\phi_2}-
\frac{\beta}{2}(2-2\cos\theta_2)^{\frac{\beta+1}{2}}e^{\mathbf{i}(\beta-1)\phi_2}\bigg)\\\nonumber
&-\frac{s}{r}e_-(x,y)\bigg((2-2\cos\theta_2)^{\frac{\beta}{2}}e^{-\mathbf{i}\beta\phi_2}-
\frac{\beta}{2}(2-2\cos\theta_2)^{\frac{\beta+1}{2}}e^{-\mathbf{i}(\beta-1)\phi_2}\bigg)\\\nonumber
=&-(2-2\cos\theta_1)^{\frac{\alpha}{2}}\bigg[d_+(x,y)\bigg(e^{\mathbf{i}\alpha\phi_1}
-\frac{\alpha}{2}(2-2\cos\theta_1)^{\frac{1}{2}}e^{\mathbf{i}(\alpha-1)\phi_1}\bigg)\\\nonumber
&+d_-(x,y)\bigg(e^{-\mathbf{i}\alpha\phi_1}
-\frac{\alpha}{2}(2-2\cos\theta_1)^{\frac{1}{2}}e^{-\mathbf{i}(\alpha-1)\phi_1}\bigg)\bigg]\\\nonumber
&-\frac{s}{r}(2-2\cos\theta_2)^{\frac{\beta}{2}}\bigg[e_+(x,y)\bigg(e^{\mathbf{i}\beta\phi_2}
-\frac{\beta}{2}(2-2\cos\theta_2)^{\frac{1}{2}}e^{\mathbf{i}(\beta-1)\phi_2}\bigg)\\
&+e_-(x,y)\bigg(e^{-\mathbf{i}\beta\phi_2}
-\frac{\beta}{2}(2-2\cos\theta_2)^{\frac{1}{2}}e^{-\mathbf{i}(\beta-1)\phi_2}\bigg)\bigg],
\end{align}
\begin{eqnarray}\label{PPP}\nonumber
\mathbf{\mathcal{F}}_{(\alpha,\beta)}(\boldsymbol\theta)&=&
-2d\,\bigg((2-2\cos\theta_1)^{\frac{\alpha}{2}}\cos(\alpha\phi_1)
-\frac{\alpha}{2}(2-2\cos\theta_1)^{\frac{\alpha+1}{2}}\cos((\alpha-1)\phi_1)\bigg)\\\nonumber
&&-2e\frac{s}{r}\,\bigg((2-2\cos\theta_2)^{\frac{\beta}{2}}\cos(\beta\phi_2)
-\frac{\beta}{2}(2-2\cos\theta_2)^{\frac{\beta+1}{2}}\cos((\beta-1)\phi_2)\bigg)\\\nonumber
&=&-2(2-2\cos\theta_1)^{\frac{\alpha}{2}}\bigg[d\,\bigg(\cos(\alpha\phi_1)
-\frac{\alpha}{2}(2-2\cos\theta_1)^{\frac{1}{2}}\cos((\alpha-1)\phi_1)\bigg)\\\nonumber
&&+e\frac{s}{r}\,\bigg(\frac{(2-2\cos\theta_2)^{\frac{\beta}{2}}}{(2-2\cos\theta_1)^{\frac{\alpha}{2}}}\cos(\beta\phi_2)
-\frac{\beta}{2}\frac{(2-2\cos\theta_2)^{\frac{\beta+1}{2}}}{(2-2\cos\theta_1)^{\frac{\alpha}{2}}}\cos((\beta-1)\phi_2)\bigg)\bigg]\\\nonumber
&=&-2(2-2\cos\theta_2)^{\frac{\beta}{2}}\bigg[d\,\bigg(\frac{(2-2\cos\theta_1)^{\frac{\alpha}{2}}}{(2-2\cos\theta_2)^{\frac{\beta}{2}}}\cos(\alpha\phi_1)
-\frac{\alpha}{2}\frac{(2-2\cos\theta_1)^{\frac{\alpha+1}{2}}}{(2-2\cos\theta_2)^{\frac{\beta}{2}}}\cos((\alpha-1)\phi_1)\\
&&+e\frac{s}{r}\,\bigg(\cos(\beta\phi_2)-\frac{\beta}{2}(2-2\cos\theta_2)^{\frac{1}{2}}\cos((\beta-1)\phi_2)\bigg)\bigg].
\end{eqnarray}
According to the definition of $\phi_k$ in \eqref{limtan}, it is easy to see that for
$\gamma\in(1,2)$ and $k=1,2$, we have
\begin{align}\label{tan:3}
\begin{split}
\limsup\limits_{\theta_k\to0}\,\tan(\gamma\phi_k)&=\lim\limits_{\theta_k\to0^+}\tan(\gamma\phi_k)=-\tan\bigg(\gamma\frac{\pi}{2}\bigg)>0,\\ \liminf\limits_{\theta_k\to0}\,\tan(\gamma\phi_k)&=\lim\limits_{\theta_k\to0^-}\tan(\gamma\phi_k)=\tan\bigg(\gamma\frac{\pi}{2}\bigg)<0.
\end{split}
\end{align}
Combining relations (\ref{HHH}) and (\ref{PPP}), we obtain
\begin{align}\label{HHHPPP}\nonumber
&\frac{h_{(\alpha,\beta)}(\mathbf{x},\boldsymbol\theta)}{\mathbf{\mathcal{F}}_{(\alpha,\beta)}(\boldsymbol\theta)}=\\\nonumber
&\frac{d_+(x,y)\bigg(e^{\mathbf{i}\alpha\phi_1}
-\frac{\alpha}{2}(2-2\cos\theta_1)^{\frac{1}{2}}e^{\mathbf{i}(\alpha-1)\phi_1}\bigg)
+d_-(x,y)\bigg(e^{-\mathbf{i}\alpha\phi_1}
-\frac{\alpha}{2}(2-2\cos\theta_1)^{\frac{1}{2}}e^{-\mathbf{i}(\alpha-1)\phi_1}\bigg)}{2\bigg[d\,\bigg(\cos(\alpha\phi_1)
-\frac{\alpha}{2}(2-2\cos\theta_1)^{\frac{1}{2}}\cos((\alpha-1)\phi_1)\bigg)
+e\frac{s}{r}\,\bigg(\frac{(2-2\cos\theta_2)^{\frac{\beta}{2}}}{(2-2\cos\theta_1)^{\frac{\alpha}{2}}}\cos(\beta\phi_2)
-\frac{\beta}{2}\frac{(2-2\cos\theta_2)^{\frac{\beta+1}{2}}}{(2-2\cos\theta_1)^{\frac{\alpha}{2}}}\cos((\beta-1)\phi_2)\bigg)\bigg]}+\\
&\frac{\frac{s}{r}e_+(x,y)\bigg(e^{\mathbf{i}\beta\phi_2}
-\frac{\beta}{2}(2-2\cos\theta_2)^{\frac{1}{2}}e^{\mathbf{i}(\beta-1)\phi_2}\bigg)
+\frac{s}{r}e_-(x,y)\bigg(e^{-\mathbf{i}\beta\phi_2}
-\frac{\beta}{2}(2-2\cos\theta_2)^{\frac{1}{2}}e^{-\mathbf{i}(\beta-1)\phi_2}\bigg)}{2\bigg[d\,\bigg(\frac{(2-2\cos\theta_1)^{\frac{\alpha}{2}}}{(2-2\cos\theta_2)^{\frac{\beta}{2}}}\cos(\alpha\phi_1)
-\frac{\alpha}{2}\frac{(2-2\cos\theta_1)^{\frac{\alpha+1}{2}}}{(2-2\cos\theta_2)^{\frac{\beta}{2}}}\cos((\alpha-1)\phi_1)
+e\frac{s}{r}\,\bigg(\cos(\beta\phi_2)-\frac{\beta}{2}(2-2\cos\theta_2)^{\frac{1}{2}}\cos((\beta-1)\phi_2)\bigg)\bigg]}
\end{align}
Now, we calculate the limit of the quotient
$\frac{h_{(\alpha,\beta)}(\mathbf{x},\boldsymbol\theta)}{\mathbf{\mathcal{F}}_{(\alpha,\beta)}(\boldsymbol\theta)}$. If $\gamma=\alpha$, we have
\begin{eqnarray}\label{alphaa}
\lim\limits_{{\boldsymbol\theta\to\mathbf{0}}}\frac{(1-\cos\theta_2)^{\beta}}{(1-\cos\theta_1)^{\alpha}}=0,~~~~~~~~~~~~
\lim\limits_{\boldsymbol\theta\to\mathbf{0}}\frac{(1-\cos\theta_1)^{\alpha}}{(1-\cos\theta_2)^{\beta}}=\infty,
\end{eqnarray}
and then thanks to relations \eqref{tan:3},
(\ref{HHHPPP}), and (\ref{alphaa}), we conclude that
\begin{align*}
\begin{split}
\limsup\limits_{\boldsymbol\theta\to\mathbf{0}}\frac{h_{(\alpha,\beta)}(\mathbf{x},\boldsymbol\theta)}{\mathbf{\mathcal{F}}_{(\alpha,\beta)}(\boldsymbol\theta)}
&=\limsup\limits_{\theta_1\to0}\frac{d_+(x,y)e^{\mathbf{i}\alpha\phi_1}
+d_-(x,y)e^{-\mathbf{i}\alpha\phi_1}}{2d\cos(\alpha\phi_1)}\\
&=\frac{d_+(x,y)+d_-(x,y)}{2d}+\mathbf{i}(\lim\limits_{\theta_1\to 0^+}\tan(\alpha\phi_1))\frac{d_+(x,y)-d_-(x,y)}{2d}.
\end{split}
\end{align*}

If $\gamma=\beta$, we obtain an analogous result with $\beta$, $e_{\pm}$, $\theta_2$ in place of $\alpha$, $d_{\pm}$, $\theta_1$, respectively, and the thesis is proved.
\end{proof}

\section{Multigrid methods}\label{sec:mgm}

Multigrid methods have shown to be a valid alternative to
preconditioned Krylov methods also for FDEs \cite{29}. Using the
Ruge-St\"{u}ben theory \cite{20}, Theorem 4 in \cite{M2} and in
\cite{29} for one-dimensional case shows that, in the constant and
varying coefficients cases, i.e., $d_{\pm}(x,t)=d>0$ and
$d_+(x,t)=d_-(x,t)>0$, respectively, the two-grid method converges
with a linear convergence rate independent of $N$ and $m$. Since
in this case the matrix
$\mathbf{\mathcal{M}}^{(m)}_{(\alpha,\beta),N}$ is a BTTB matrix,
the classical multigrid theory for BTTB matrices developed in
\cite{9,6,2424,222} can be directly applied when the symbol is
known. Under the assumptions that $d_{\pm}(x,y,t)=d>0$,
$e_{\pm}(x,y,t)=e>0$, $\frac{1}{r}=o(1)$ and
$\frac{s}{r}=\frac{h^{\alpha}_x}{h^{\beta}_y}=\mathcal{O}(1)$, according
to our previous analysis in Subsection \ref{CDC}, the symbol of the
BTTB sequence
$\{\mathbf{\mathcal{M}}^{(m)}_{(\alpha,\beta),N}\}_{_{N\in
\mathbb{N}}}$
 is $d\, q_{\alpha}(\theta_1)+\frac{s}{r}e\, q_{\beta}(\theta_2)$ (cf.
Proposition \ref{PaPb}).

Consider the stationary iterative method
\begin{eqnarray}\label{Sn}
x^{(j+1)} = S_Nx^{(j)} + b_1 := \mathcal{S}_N(x^{(j)}, b_1)
\end{eqnarray}
for the solution of the linear system $A_Nx=b$ where
$A_N,W_N,S_N:=I-W_N^{-1}A_N\in\mathbb{C}^{N\times N}$, and
$b,b_1:=W_N^{-1}b\in\mathbb{C}^N$. Given a full-rank matrix
$P_N\in\mathbb{C}^{N\times k}$, with $k<N$, a \emph{Two-Grid Method
(TGM)} is defined by the following algorithm~\cite{18}.
\begin{center}
\begin{tabular*}{\columnwidth}{@{\extracolsep{\fill}}*{1}{l}}
\hline \hline
\textbf{TGM}$(S_N,P_N)$\\
\hline \hline
1) $r_N=b-A_Nx^{(j)}$\\
2) $d_k=P_N^Tr_N$\\
3) $A_k=P_N^TA_NP_N$\\
4) Solve $A_ky=d_k$\\
5) $\hat{x}^{(j)}=x^{(j)}+P_Ny$\\
6) $x^{(j+1)}=\mathcal{S}^{\nu}_N(\hat{x}^{(j)}, b_1)$\\
\hline \hline
\end{tabular*}
\end{center}
Step 6) consists in applying the ``post-smoothing
iteration'' (\ref{Sn}) $\nu$ times while steps 1)$\rightarrow$ 5)
define the ``coarse grid correction" which depends on the
grid transfer operator $P_N$. The iteration matrix of
TGM is then given by
\[TGM =
S^{\nu}_N\bigg[I-P_N\bigg(P_N^TA_NP_N\bigg)^{-1}P_N^TA_N\bigg].
\]
It could be possible also to add a ``pre-smoothing'' iteration in the TGM algorithm before the step 1).
We do not add it here for keeping the theoretical analysis as simple as possible. Nevertheless, the addition of a convergent pre-smoother, obviously, cannot deteriorate the convergence of the algorithm but only accelerate the convergence.

Here we want to give a formal proof of convergence and of the
optimality of TGM. 
We recall some convergence results in
\cite{20} from the theory of the algebraic multigrid method which
are the main theoretical tools for giving our convergence proof. By
$\|\cdot\|_2$ we denote the Euclidean norm and whenever $X$
is positive definite, $\|\cdot\|_X=\|X^{1/2}\cdot\|_2$ denotes the
energy norm of $X$. Finally if $X$ and $Y$ are Hermitian matrices,
then $X\leq Y$ is equivalent to write that $Y-X$ is nonnegative
definite.
\begin{theorem}\label{theorem21}
\cite{20} Let $A_N$ be a positive definite matrix of size $N$ and
let $S_N$ be defined as in the TGM algorithm. Suppose that
$\exists\;\delta>0$ independent of $N$ such that
\begin{eqnarray}\label{smo}
\|S_N\mathbf{x}_N\|^2_{A_N}\leq\|\mathbf{x}_N\|^2_{A_N}-\delta\|\mathbf{x}_N\|^2_{A_ND^{-1}_NA_N},
~~~~~~~\forall\mathbf{x}_N\in\mathbb{C}^N,
\end{eqnarray}
where $D_N$ is the diagonal matrix having the same diagonal of $A_N$. Assume that  $\exists \, \xi>0$
independent of $N$ such that
\begin{eqnarray}\label{app}
\min\limits_{\mathbf{y}\in\mathbb{C}^k}\|\mathbf{x}_N-P_N\mathbf{y}\|^2_{D_N}\leq\xi\|\mathbf{x}_N\|^2_{A_N},
~~~~~~~\forall\mathbf{x}_N\in\mathbb{C}^N.
\end{eqnarray}
Then $\xi\geq\delta$ and
\[
\|TGM\|_{A_N}\leq\sqrt{1-\delta/\xi}.
\]
\end{theorem}

When $N$ is huge also the coarser problem can be too large to be solved directly. Hence, the point 4) in the TGM should be replaced by a recursive application of the same strategy until we reach a sufficiently small size.
A reliable initial guess for the coarser linear system is the zero vector. Indeed, at the coarser levels the linear system to be solved in the error equation and hence the solution goes to zero whenever the whole iteration converges. The resulting algorithm is usually referred as V-cycle.

\subsection{Two-grid convergence analysis}\label{sec:tgmconv}
The convergence analysis of the TGM is firstly developed in the constant coefficient case, i.e., in the case $A_N = T_N^{(2)}(\mathbf{\mathcal{F}}_{(\alpha,\beta)})$.
The variable coefficient case will be discussed at the end of the subsection.

Let us start showing that damped Jacobi, with a proper choice of the relaxation parameter, satisfies the \emph{smoothing property} \eqref{smo}.

\begin{lem}\label{lemsmo}
Let $A_N:=T^{(2)}_{N}(\mathbf{\mathcal{F}}_{(\alpha,\beta)})$ with $\mathbf{\mathcal{F}}_{(\alpha,\beta)}$
defined as in \eqref{P} and let
$S_N:=I_N-\omega D^{-1}_NA_N$, where $\omega$ is a real number and $D_N=a_0I_N$ with $a_0$ the Fourier coefficient of $\mathbf{\mathcal{F}}_{(\alpha,\beta)}$ of order zero. Moreover,
let us assume that $\frac{1}{r}=\frac{2h_x^{\alpha}}{\triangle t}=o(1)$ and
$\frac{s}{r}=\frac{h_x^{\alpha}}{h_y^{\beta}}=\mathcal{O}(1)$. If we
choose $0<\omega<\frac{2a_0}{\|\mathbf{\mathcal{F}}_{(\alpha,\beta)}\|_{\infty}}$,
then $\exists\;\delta>0$ such that inequality (\ref{smo}) holds true. 
\end{lem}
\begin{proof}
By recalling that $\mathbf{\mathcal{F}}_{(\alpha,\beta)}$ is nonnegative, it follows
that $a_0=\frac{\|\mathbf{\mathcal{F}}_{(\alpha,\beta)}\|_1}{2\pi}$. Hence the
relation given in (\ref{smo}) is equivalent to the inequality
\[
\left(I_N-\omega D^{-1}_NA_N\right)A_N\left(I_N-\omega
D^{-1}_NA_N\right)\leq A_N-\delta A_ND^{-1}_NA_N,
\]
which can be rewritten as
\begin{equation}\label{ll}
\left(I_N-\omega D^{-1}_NA_N\right)^2\leq I_N-\delta
A^{\frac{1}{2}}_ND^{-1}_NA^{\frac{1}{2}}_N,
\end{equation}
using a congruence transformation with $A^{-\frac{1}{2}}$.
Since $D_N=a_0I_N$, where $a_0=\frac{1}{r}-2(d\,
w^{(\alpha)}_1+\frac{s}{r}e\,w^{(\beta)}_1)$ and $a_0>0$ (cf.~(\ref{ww})),
under the hypothesis that
$\frac{1}{r}=\frac{2h_x^{\alpha}}{\triangle t}=o(1)$ and
$\frac{s}{r}=\frac{h_x^{\alpha}}{h_y^{\beta}}=\mathcal{O}(1)$, we
have that 
$a_0$ is independent of $N$. The inequality (\ref{ll}) reads as
\[
\left(I_N-\frac{\omega}{a_0}A_N\right)^2\leq
I_N-\left(\frac{\delta}{a_0}\right)A_N,
\]
which is implied by the function inequality
$(1-\frac{\omega}{a_0}\mathbf{\mathcal{F}}_{(\alpha,\beta)})^2\leq 1-(\frac{\delta}{a_0})\mathbf{\mathcal{F}}_{(\alpha,\beta)}$.
From the latter inequality, if $0<\omega < \frac{2a_0}{\|\mathbf{\mathcal{F}}_{(\alpha,\beta)}\|_{\infty}}$
then it exists $\delta>0$ such that equation \eqref{smo} holds and the thesis is proved (see
Proposition 3 in \cite{0} for more details).
\end{proof}

In order to prove the \emph{approximation property} \eqref{app}, we define the projector $P_N$ as
\[P_N= T^{(2)}_N(p)U^k_N,\]
where the function $p$ will be defined later and $U^k_N = K^{k_1}_{n_1}\otimes K^{k_2}_{n_2}$ is the bidimensional down-sampling operator.
Let $k_\ell = (n_\ell-(n_\ell \, \mathrm{mod} \, 2))/2$ with $\ell=1,2$, the one-dimensional down-sampling matrix $K^k_n \in \mathbb{R}^{n
\times k}$ is defined as
\[
    [K^k_n]_{i,j} =
    \left\{
    \begin{array}{l l}
        1 & \mbox{if } i = 2j - (n+1) \, \mathrm{mod} \, 2, \\
        0 & \mbox{otherwise,}
    \end{array}
    \right.
    \qquad
    j = 1,\dots,k.
\]
Lemma \ref{lemapp} gives the theoretical conditions that $p$ has to meet in order to satisfy \eqref{app}.
A preliminary proposition is necessary to extend the theoretical analysis usually done matrix algebras, like circulant or $\tau$ matrices (see e.g. \cite{0,222}) also to (multilevel) Toeplitz matrices.
\begin{pro}{\cite{27}}\label{pro21}
Given a matrix $A$, let us denote by $\rm diag(A)$ the diagonal matrix having the same diagonal of $A$.
Moreover, let us assume that the property (\ref{app}) is fulfilled by a matrix-sequence
$\{A_N\}_N$. If there exists another matrix-sequence $\{B_N\}_N$ such that
\begin{equation}\label{eq:AB1}
A_N\leq\vartheta B_N,
\end{equation}
\begin{equation}\label{eq:AB2}
\rm diag(A_N)\geq\eta\,\rm diag(B_N),
\end{equation}
$\eta,\vartheta>0$, then property (\ref{app}) is fulfilled by $\{B_N\}_N$ with the same $P_N$ as well.
\end{pro}

\begin{lem}\label{lemapp}
Let $A_N:=T^{(2)}_{N}(\mathbf{\mathcal{F}}_{(\alpha,\beta)})$ with
$\mathbf{\mathcal{F}}_{(\alpha,\beta)}$
defined as in \eqref{P}
and $n_\ell = 2k_\ell-1$ for $\ell=1,2$.
Moreover, let
$P_N=T^{(2)}_{N}(p)U^k_N$ with $p$ the following polynomial
\begin{equation}\label{eq:q}
p(\theta_1,\theta_2)=(1+\cos(\theta_1))(1+\cos(\theta_2)),
\end{equation}
and assume that $\frac{1}{r}=\frac{2h_x^{\alpha}}{\triangle t}=o(1)$ and
$\frac{s}{r}=\frac{h_x^{\alpha}}{h_y^{\beta}}=\mathcal{O}(1)$. Then
there exists $\xi>0$ such that relation (\ref{app})
holds true.
\end{lem}
\begin{proof}
The proof combines classical results on multigrid methods for Toeplitz matrices with the spectral results in Section \ref{sec:spectr}.
According to Proposition \ref{PaPb}, the function $\mathbf{\mathcal{F}}_{(\alpha,\beta)}$ is nonnegative and vanishes only at the origin.
Moreover, thanks to Proposition \ref{pro4}, the polynomial $p$ in \eqref{eq:q} satisfies the classical condition
\begin{eqnarray}\label{marco}
\limsup\limits_{\mathbf{x}\to(0,0)}\frac{p(\mathbf{y})^2}{\mathbf{\mathcal{F}}_{(\alpha,\beta)}(\mathbf{x})}=c<+\infty,
~~~~~~~~~~\forall\mathbf{y}\in M(\mathbf{x}),
\end{eqnarray}
with $M(\mathbf{x})$ being the set of the ``mirror points" of
$\mathbf{x}$ defined as $M(\mathbf{x})=\{(x_1,\pi-x_2),(\pi-x_1,x_2),(\pi-x_1,\pi-x_2)\}$,
see \cite{2424} for the derivation of condition \eqref{marco}.
In particular, the condition \eqref{marco} holds with $c=0$.
Therefore, the TGM defined in the two-level $\tau$ algebra would be convergent.

The result for BTTB matrices follows from Proposition \ref{pro21}.
Note that $T^{(2)}_{N}(p)$ is a two-level $\tau$ matrix and $n_\ell = 2k_\ell-1$ for $\ell=1,2$. Hence the projector
$P_N$ is exactly the same projector used in TGM defined in the two-levels $\tau$ algebra, cf. \cite{2424}.
Finally, Proposition \ref{pro21} can be applied replacing the sequences $\{A_N\}_N$ and $\{B_N\}_N$ with the sequences of two-levels $\tau$ and BTTB matrices generated by $\mathbf{\mathcal{F}}_{(\alpha,\beta)}$, respectively.
Indeed the condition \eqref{eq:AB1} follows from Theorem 7.1 in  \cite{2424}, while the condition \eqref{eq:AB2}
is a trivial consequence of Proposition \ref{ww} with $\eta=1$.
\end{proof}

\begin{rem}
The grid transfer operator associated to the symbol $p$ defined as in \eqref{eq:q}
is, up to a constant factor,  the classical bilinear interpolation.
\end{rem}

By Lemmas \ref{lemsmo} and \ref{lemapp}, it follows that there exist
$\delta$ and $\xi$ such that inequalities (\ref{smo}) and
(\ref{app}) hold, respectively. Therefore, by Theorem
\ref{theorem21}, $\xi\geq\delta$ and
$\|TGM(S_N,P_N)\|_{A_N}\leq\sqrt{1-\delta/\xi}$.
The result can be summarized as follows.
\begin{theorem}\label{theorem21bis}
Let $A_N:=T^{(2)}_{N}(\mathbf{\mathcal{F}}_{(\alpha,\beta)})$ with $\mathbf{\mathcal{F}}_{(\alpha,\beta)}$
defined as in \eqref{P}.
\begin{itemize}
\item
Let $S_N:=I_N-\omega D^{-1}_NA_N$, where $0<\omega<\frac{2a_0}{\|\mathbf{\mathcal{F}}_{(\alpha,\beta)}\|_{\infty}}$ and $D_N=a_0I_N$ with $a_0$ the Fourier coefficient of $\mathbf{\mathcal{F}}_{(\alpha,\beta)}$ of order zero (Jacobi smoother).
\item
Let
$P_N=T^{(2)}_{N}(p)U^k_N$ with $p$ defined as in \eqref{eq:q} (bilinear interpolation).
\end{itemize}
Moreover, let us
assume that $\frac{1}{r}=\frac{2h_x^{\alpha}}{\triangle t}=o(1)$ and
$\frac{s}{r}=\frac{h_x^{\alpha}}{h_y^{\beta}}=\mathcal{O}(1)$.
Then the assumptions of Theorem \ref{theorem21} are satisfied and it holds
\[
\|TGM(S_N,P_N)\|_{A_N}\leq c <1,
\]
with $c$ a constant independent of $N, \alpha$, and $\beta$.
\end{theorem}

The variable coefficients case can be addressed thanks to the
extension of the previous results given in \cite{2424}. Let
$d_{\pm}$ and $e_{\pm}$ be four uniformly bounded and positive
functions. Then the linear convergence rate of the two-grid method
is preserved combining Proposition \ref{pro6} with Lemma 6.2 in
\cite{2424}.

\subsection{V-cycle and geometric mulgrid}\label{sec:vgeom}

The convergence analysis of the V-cycle is much more involved and a
linear convergence rate has been proven, under a condition stricter than (\ref{marco}), only for sequences of
matrices in some trigonometric algebras, like the $\tau$ algebra used in the proof of Lemma \ref{lemapp}, see \cite{222}. In details, the symbol
$p$ of the grid transfer operator has to satisfy
\begin{eqnarray}\label{Vmarco}
\limsup\limits_{\mathbf{x}\to (0,0)}\frac{p(\mathbf{y})}{\mathbf{\mathcal{F}}_{(\alpha,\beta)}(\mathbf{x})}=c<+\infty,
~~~~~~~~~~\forall\mathbf{y}\in M(\mathbf{x}).
\end{eqnarray}
With respect to \eqref{marco}, the numerator does not have the power two and hence $p$ has to vanish at the mirror points with double order.
\begin{rem}
Similarly to Lemma \ref{lemapp}, for $p$ defined as in \eqref{eq:q}
the condition \eqref{Vmarco} holds true with $c=0$.
\end{rem}
The previous remark
suggests that the bilinear interpolation is powerful enough to work
also under some perturbations.
In particular, we could use the geometric multigrid instead of the Galerkin approach.
The geometric multigrid defines the coarser matrix $A_k$ rediscretizing the original differential operator on the coarser grid instead of computing $A_k$ by $A_k=P_N^TA_NP_N$.
This saves some computational costs with respect to Galerkin approach which requires the computation of the
coarser matrices at each recursion level in a precomputing phase.
In particular, in the variable coefficient case the coarser matrices lose the GLT structure in terms of diagonal and BTTB matrices. Therefore we should memorize all the coefficients of the two-level lower Hessenberg coefficient matrices and the computational cost of the precomputing phase is much higher than the $\mathcal{O}(N \log N)$ arithmetic operations required for the matrix-vector product with the matrix $\mathbf{\mathcal{M}}^{(m)}_{(\alpha,\beta),N}$.

In conclusion, the Galerkin approach is useful for the theoretical analysis in Subsection \ref{sec:tgmconv}, but it is not computationally feasible. Therefore, we implement the geometric approach and we denote by $\mathbf{\mathcal{P}}^{(m)}_{{\rm MGM},N}$ one iteration of the geometric multigrid algorithm applied to the linear system with coefficient matrix ${\cal M}_{(\alpha,\beta),N}^{(m)}$ and defined by the following setting:
\begin{itemize}
\item V-cycle with a number of recursion levels depending on the size $N$ (coarsest grid fixed to $8 \times 8$).
\item The smoother is one step of pre- and post-smoother for the classical Jacobi method (according to Lemma~\ref{lemsmo}).
\item The grid transfer operator is the standard bilinear interpolation and full-weighting corresponding to a proper scaling of the symbol $p$ defined as in \eqref{eq:q} (according to Lemma \ref{lemapp}).
\end{itemize}

\subsection{Preconditioning}

According to the theoretical analysis in Subsection \ref{sec:tgmconv} and the discussion at the end of Subsection \ref{sec:vgeom}, our geometric multigrid
can be effectively used as a stand alone solver. Nevertheless, in order to increase their robustness, multigrid methods are often applied as preconditioners for Krylov methods. Moreover, the application of multigrid methods as preconditioners allows a simple comparison with existing strategies based on (inverse) band or circulant preconditioners, see e.g., \cite{M1,15, 35}.

A further possibility is the combination of a band preconditioner with a multigrid method based on the Galerkin approach. The resulting strategy keeps the computational cost of the precomputing phase to $\mathcal{O}(N)$.
A simple band preconditioner suitable for Galerkin multigrid methods is the Laplacian preconditioner ($\alpha=\beta=2$). Such preconditioner is inspired by the 1D analysis in \cite{M2}, where the use of the Laplacian was introduced for fractional derivatives of order greater than 1.5.
In this paper, independently of the values of $\alpha$ and $\beta$, we propose the preconditioner
\begin{eqnarray*}
\begin{cases}
P^{(m)}_x=D^{(m)}_+(I_{n_2}\otimes
A^{2}_{n_1})+D^{(m)}_-(I_{n_2}\otimes (A^{2}_{n_1})^T),\cr
P^{(m)}_y=E^{(m)}_+(A^{2}_{n_2}\otimes
I_{n_1})+E^{(m)}_-((A^{2}_{n_2})^T\otimes I_{n_1}),\cr
\mathbf{\mathcal{P}}^{(m)}_{2,N}=\frac{1}{r}I_{N}-\bigg(P^{(m)}_x+\frac{s}{r}P^{(m)}_y\bigg),
\end{cases}
\end{eqnarray*}
where $A^{2}_{n_j}$, for $j=1,2$ is the Laplacian matrix.
The subscript $2$ in $\mathbf{\mathcal{P}}^{(m)}_{2,N}$ recall the order of the derivative used in the preconditioner.
Note that the matrix $\mathbf{\mathcal{P}}^{(m)}_{2,N}$ has only five nonzero diagonals.
On the other hand, due to fill-in, the direct solution of the linear system with coefficient matrix $\mathbf{\mathcal{P}}^{(m)}_{2,N}$ is not feasible. Therefore, instead of solving the corresponding linear system, we apply only one step of our multigrid consisting of the standard Galerkin approach with Jacobi smoothing and bilinear interpolation as grid transfer operator. Note that Theorem \ref{theorem21bis} holds also with $\alpha=\beta=2$ and the resulting preconditioner has a computational cost linear in $N$.

The goodness of this preconditioner depends on the values of $\alpha$ and $\beta$.
To define a preconditioner with the same structure of the matrix
$\mathbf{\mathcal{M}}^{(m)}_{(\alpha,\beta),N}$ and keeping at the same time a
small bandwidth, the symbol of a banded BTTB matrix has to be a
trigonometric polynomial and hence the zero of the symbol cannot be
of fractional order.
Nevertheless, the condition number of the preconditioned matrix $(\mathbf{\mathcal{P}}^{(m)}_{2,N})^{-1}{\cal M}_{(\alpha,\beta),N}^{(m)}$ is asymptotical to $N^{\frac{2-\gamma}{2}}$, with $\gamma=\max\{\alpha,\beta\}$ s.t. $0<2-\gamma<1$, 
and hence the number of iterations of a conjugate gradient type method grows as $\mathcal{O}(N^{\frac{2-\gamma}{4}})$, see~\cite{Axel}. In conclusion the preconditioner     $\mathbf{\mathcal{P}}^{(m)}_{2,N}$ is a good choice whenever $\alpha$ or $\beta$ are close to $2$, as confirmed by the numerical results in the next section.


\section{Numerical results}\label{Example}
In this section we test the effectiveness of our new multigrid
preconditioners $\mathbf{\mathcal{P}}^{(m)}_{2,N}$ and  $\mathbf{\mathcal{P}}^{(m)}_{{\rm MGM},N}$ defined in the previous section. Since two of the three proposed examples are taken from \cite{M1} we compare the performances of our preconditioners with the proposal therein (same Krylov method and same tolerance are used).

In the following tables, according to the notation defined in Section \ref{Introduction}, $n_1$ and $n_2$
denote the numbers of spatial partitions in $x$-direction and
$y$-direction, respectively, while $M$ denotes the number
of time steps.
In all our examples we simply fix $n_1=n_2=M$.
The infinite norm of the difference between the exact solution and the numerical solution at the last time step is denoted by ``$Error$".
The number of iterations is computed as
the arithmetic average of the number of iterations required for solving (\ref{2}) at each time step $t^{(m)}$.
For this reason our multigrid preconditioners will be simply denoted by $\mathbf{\mathcal{P}}_{2}$ and  $\mathbf{\mathcal{P}}_{{\rm MGM}}$.
We use the GMRES method to solve the discretized linear systems (\ref{2}).
The GMRES method is computationally performed using the built-in \textsf{gmres} Matlab function with tolerance
$10^{-7}$ with a restarting every 20 iterations, even if for the preconditioned iterations it is not necessary.
The initial guess at each time step is chosen as the zero
vector. Our computations are performed using Matlab 7.10 software
on a Pentium IV, 3.50 GHz CPU machine with 4 Gbyte of memory.

\subsection{Example 1}\label{Ex2}
The first example is taken from \cite{PS2}.
We consider a FDE of type (\ref{1}) with
$\alpha=1.8,~\beta=1.6$. The nonconstant diffusion coefficients are
given by
\begin{eqnarray*}
&&d_+(x,y,t)=\Gamma(3-\alpha)(1+x)^{\alpha}(1+y)^2,~~~~~~~d_-(x,y,t)=\Gamma(3-\alpha)(3-x)^{\alpha}(3-y)^2,\\
&&e_+(x,y,t)=\Gamma(3-\beta)(1+x)^2(1+y)^{\beta},~~~~~~~e_-(x,y,t)=\Gamma(3-\beta)(3-x)^2(3-y)^{\beta}.
\end{eqnarray*}
The spatial domain is $\Omega=[0,2]\times[0,2]$, while the time
interval is $[0,T]=[0,1]$. The initial condition is
\[
u(x,y,0)=u_0(x,y)=x^2y^2(2-x)^2(2-y)^2\\
\]
and the source term is such that the solution to the FDE is given by $u(x,y,t)=16{\rm
e}^{-t}x^2(2-x)^2y^2(2-y)^2$.

Let $h = h_x=h_y = 2/(n+1)$, with $n=n_1=n_2=M$.
Therefore, we have
\[
\frac{1}{r}=\frac{2h^{\alpha}}{\triangle
t}=\frac{2^{\alpha+1}M}{(n+1)^{\alpha}}=\frac{2^{\alpha+1}n}{(n+1)^{\alpha}},
\qquad
\frac{s}{r}=\frac{h^{\alpha}}{h^{\beta}}=
2^{\alpha-\beta}(n+1)^{\beta-\alpha}.
\]
Since $\alpha=1.8$ and $\beta=1.6$ we obtain $\frac{1}{r}=O(n^{-0.8})$
and $\frac{s}{r}=O(n^{-0.2})$. Then when $n\rightarrow\infty$, both $\frac{1}{r}$ and $\frac{s}{r}$ tend to zero. As a consequence, the term in $\beta$ goes slowly to zero and for large $n$ the term in $\alpha$ becomes dominant.

Table \ref{tabel3} compares the iterations and the accuracy of numerical solutions provided by the GMRES without preconditioning and by the GMRES preconditioned with our proposals. For comparison we report also the number of iterations required by the ``exact'' preconditioners $\widetilde{\mathbf{\mathcal{P}}}_{2}$ and $\widetilde{\mathbf{\mathcal{P}}}_{{\rm MGM}}$, where $\widetilde{\mathbf{\mathcal{P}}}_{2}$ denotes the
preconditioner $\mathbf{\mathcal{P}}_{2}$ with the direct solution of the resulting linear system
and $\widetilde{\mathbf{\mathcal{P}}}_{{\rm MGM}}$
denotes the proposed multigrid applied to the matrix $\mathbf{\mathcal{M}}^{(m)}_{(\alpha,\beta),N}$, but using the Galerkin approach instead of the rediscretization.
Note that the geometric approach is as effective as the Galerkin approach ($\mathbf{\mathcal{P}}_{{\rm MGM}}$ and $\widetilde{\mathbf{\mathcal{P}}}_{{\rm MGM}}$ provide the same number of iterations). On the contrary, one step of V-cycle for the preconditioner $\mathbf{\mathcal{P}}_{2}$ is not as good as its direct inversion, i.e., $\widetilde{\mathbf{\mathcal{P}}}_{2}$. Nevertheless, it shows a linear convergence rate and it is computationally very cheap since it is a pentadiagonal matrix. A good compromise could be to increase the number of iterations of the preconditioner $\mathbf{\mathcal{P}}_{2}$, for instance, fixing $n=2^7$ and applying two steps of V-cycle instead of one, the number of iterations reduces from 17 to 14.


\begin{table}
\caption{Comparison of average number of iterations for the GMRES methods with
different preconditioners for Example \ref{Ex2}.}
\begin{tabular*}{\columnwidth}{@{\extracolsep{\fill}}*{15}{c}}
\toprule\toprule
\multicolumn{1}{c}{}
$n_1=n_2$ &    GMRES & $\mathbf{\mathcal{P}}^{(m)}_{2,N}$  & $\widetilde{\mathbf{\mathcal{P}}}^{(m)}_{2,N}$  &  $\mathbf{\mathcal{P}}^{(m)}_{{\rm MGM},N}$ &  $\widetilde{\mathbf{\mathcal{P}}}^{(m)}_{{\rm MGM},N}$ & $Error$ \\
\midrule
$2^4$ &  $37.000$ & $21.000$ &9.000&  $10.000$& $10.000$ &  $9.3706 \times 10^{-2}$ \\
$2^5$ &  $73.000$ & $18.781$ &9.000 &  $11.000$ & $11.000$&  $2.4747 \times 10^{-2}$ \\
$2^6$ &  $137.000$ &  $17.000$ & 9.000 & $11.000$& $11.000$ &  $6.3630 \times 10^{-3}$ \\
$2^7$ &  $251.000$ & $17.000$ &10.000&  $10.000$& $10.000$ &  $1.6053 \times 10^{-3}$ \\
\bottomrule\bottomrule
\end{tabular*}
\label{tabel3}
\end{table}

\subsection{Example 2}\label{Ex1}
This example is taken from \cite{M1}.
We consider a FDE of type (\ref{1}) with
$\alpha=1.8,~\beta=1.9$. The nonconstant diffusion coefficients are
given by
\begin{eqnarray*}
&&d_+(x,y,t)=4(1+t)x^{\alpha}(1+y),~~~~~~~d_-(x,y,t)=4(1+t)(1-x)^{\alpha}(1+y),\\
&&e_+(x,y,t)=4(1+t)(1+x)y^{\beta},~~~~~~~e_-(x,y,t)=4(1+t)(1+x)(1-y)^{\beta}.
\end{eqnarray*}
The spatial domain is $\Omega=[0,1]\times[0,1]$, while the time
interval is $[0,T]=[0,1]$. The initial condition is
\[
u(x,y,0)=u_0(x,y)=x^3y^3(1-x)^3(1-y)^3
\]
and the source term is such that the solution to the FDE is
given by $u(x,y,t)={\rm e}^{-t}x^3(1-x)^3y^3(1-y)^3$.

Note that
\[
\frac{1}{r}=\frac{2h^{\alpha}}{\triangle
t}=\frac{2M}{(n+1)^{\alpha}}=\frac{2n}{(n+1)^{\alpha}},
\qquad
\frac{s}{r}=\frac{h^{\alpha}}{h^{\beta}}=\frac{(n+1)^{\beta}}{(n+1)^{\alpha}}=(n+1)^{\beta-\alpha}.
\]
Since $\alpha=1.8$ and $\beta=1.9$  we obtain $\frac{1}{r}=O(n^{-0.8})$
and $\frac{s}{r}=O(n^{0.1})$, i.e., when $n\rightarrow\infty$,
$\frac{1}{r}$ goes to zero and $\frac{s}{r}$ tends to infinity. On the other hand, $\frac{s}{r}$ grows very slowly and the numerical results are not affected by this small grow.

In Table \ref{tabel1} we compare the average number of iterations obtained by the GMRES and by our preconditioners with the performances of the preconditioner proposed in \cite{M1} and denoted by $P_{JLZ}$. Moreover, we show the accuracy of the numerical solution. We observe that in this example the average number of iterations for the preconditioner $P_{JLZ}$ grows with the algebraic size of the problem. Conversely, our preconditioners show a linear convergence rate with a lower computational cost per iteration, especially the preconditioner
$\mathbf{\mathcal{P}}_{2}$ which has only five nonzero diagonals and applies only two Jacobi iterations on each grid.


\begin{table}
\caption{Comparison of average number of iterations for the GMRES methods with different preconditioners for Example \ref{Ex1}. The preconditioner $P_{JLZ}$ has been proposed in \cite{M1}.}
\begin{tabular*}{\columnwidth}{@{\extracolsep{\fill}}*{15}{c}}
\toprule\toprule
\multicolumn{1}{c}{}
$n_1=n_2$ &    GMRES & $P_{JLZ}$ & $\mathbf{\mathcal{P}}_{2}$ & $\mathbf{\mathcal{P}}_{{\rm MGM}}$ &$Error$ \\
\midrule
$2^4$ &  $48.750$ & $11.000$ &$18.063$ & $9.000$ &  $1.1386 \times 10^{-6}$  \\
$2^5$ &  $81.594$ & $12.406$ &  $15.813$ & $9.000$ &  $3.0206 \times 10^{-7}$ \\
$2^6$ &  $157.750$ & $14.250$ & $11.531$ & $10.000$ &  $7.6925 \times 10^{-8}$ \\
$2^7$ &  $273.914$ & $17.055$ &  $12.000$ & $9.891$ &  $1.9392 \times 10^{-8}$ \\
\bottomrule\bottomrule
\end{tabular*}
\label{tabel1}
\end{table}

\subsection{Example 3}\label{Ex3}
Also this third example is taken from \cite{M1} and $\alpha,\beta$ are the same as in Example 2. The nonconstant diffusion coefficients are given by
\[
d_+(x,y,t)=6x^{\alpha}, \qquad d_-(x,y,t)=6(1-x)^{\alpha}, \qquad
e_+(x,y,t)=6y^{\beta}, \qquad e_-(x,y,t)=6(1-y)^{\beta}.
\]
The spatial domain is $\Omega=[0,1]\times[0,1]$, while the time
interval is $[0,T]=[0,1]$. The initial condition is
\[
u(x,y,0)=u_0(x,y)=x^3y^3(1-x)^3(1-y)^3
\]
and the source term is such that the solution of the FDE is
given by $u(x,y,t)={\rm e}^{-t}x^3(1-x)^3y^3(1-y)^3$.

The behaviour of $\frac{1}{r}$ and $\frac{s}{r}$ is the same of the previous example.
The results in Table \ref{tabel4} are comparable to those of the previous example and again our preconditioners lead to a fast convergence with a number of iterations independent of the size $N$.

\begin{table}
\caption{Comparison of average number of iterations for the GMRES methods with different preconditioners for Example \ref{Ex3}. The preconditioner $P_{JLZ}$ has been proposed in \cite{M1}.}
\begin{tabular*}{\columnwidth}{@{\extracolsep{\fill}}*{15}{c}}
\toprule\toprule
\multicolumn{1}{c}{}
$n_1=n_2$ &    GMRES & $P_{JLZ}$ & $\mathbf{\mathcal{P}}_{2}$ & $\mathbf{\mathcal{P}}_{{\rm MGM}}$ &$Error$ \\
\midrule
$2^4$ &  $36.000$ & $11.000$ & $17.000$ & $9.000$ &  $1.1486 \times 10^{-6}$  \\
$2^5$ &  $63.694$ & $12.250$ &  $15.000$ & $9.000$ &  $2.9187 \times 10^{-7}$ \\
$2^6$ &  $113.234$ & $13.813$ & $11.000$ & $8.000$ &  $7.4461 \times 10^{-8}$ \\
$2^7$ &  $173.008$ & $15.465$ &  $11.000$ & $8.000$ &  $1.8782 \times 10^{-8}$ \\
\bottomrule\bottomrule
\end{tabular*}
\label{tabel4}
\end{table}

\section{Conclusions}\label{Conclusions}
In this paper we have investigated two-dimensional space-FDE problems discretized by means of a second order finite difference scheme obtained as combination of the Crank-Nicolson scheme and the so-called weighted and shifted Gr\"unwald formula. We have provided a detailed spectral analysis of the coefficient matrix by means of its symbol, both in the constant and variable coefficients case.
Thanks to the symbol analysis and the theory of multigrid methods for BTTB matrices, we have designed a classical multigrid method particularly effective for the considered problem. Two multigrid preconditioners for GMRES has been proposed and some numerical examples taken from the literature show that they provide a fast convergence independent of the algebraic size of the problem.

Multigrid methods have shown a good scalability in the dimensionality of the problem and they could be a good choice also for 3D and 4D problems since their extension is straightforward. Moreover, they could be effectively applied also to other discretization strategies, like the finite volume method proposed in \cite{Ng}.

\section*{Acknowledgment}
The work of the last two authors is partly supported by the Italian
grant MIUR - PRIN 2012 N. 2012MTE38N and by GNCS-INDAM (Italy).

\bibliographystyle{elsart}

\end{document}